\documentclass[12pt,reqno]{amsart}
\usepackage{amsfonts,amstext}
\usepackage{amsmath,amstext,amssymb,amscd,mathrsfs,amsthm}
\usepackage{enumerate}
\usepackage[dvipsnames,usenames]{xcolor}

\newtheorem{thm}{Theorem}[section]
\newtheorem{lem}[thm]{Lemma}
\newtheorem{cor}[thm]{Corollary}
\newtheorem{prop}[thm]{Proposition}
\newtheorem{assumption}[thm]{Assumption}
\theoremstyle{definition}
\newtheorem{defn}[thm]{Definition}

\theoremstyle{remark}
\newtheorem{rmk}[thm]{Remark}
\numberwithin{equation}{section}

\usepackage[allbordercolors={1 1 1}]{hyperref}
\usepackage[disable]{todonotes}

\def\grad{\nabla}

\renewcommand{\H}[1][1]{W^{#1,2}(\Omega)}

\renewcommand{\to}{\mathrel{\rightarrow}}
\newcommand{\tinto}{\mathrel{\overset{\gamma}{\to}}}

\newcommand{\R}{\mathbb{R}}
\newcommand{\N}{\mathbb{N}}
\newcommand{\E}{\mathscr{E}}

\def\g{\gamma}

\def\s{\sigma}
\def\t{\tau}

\def\G{\Gamma}
\def\O{\Omega}

\def\e{\epsilon}
\def\p{\phi}

\newcommand{\norm}[1]{\left\|#1\right\|}
\newcommand{\abs}[1]{\left|#1\right|}

\def\N{\mathbb{N}}
\def\lra{\longrightarrow}
\def\<{\langle}
\def\>{\rangle}
\def\lra{\longrightarrow}
\def\ra{\rightarrow}
\def\H{H^1_{\G_0}}
\begin{document}

\author[A. R. Becklin]{Andrew R. Becklin}
\address{Department of Mathematics, University of Nebraska--Lincoln, Lincoln, NE  68588-0130, USA} \email{andrew.becklin@huskers.unl.edu}

\author[M. A. Rammaha]{Mohammad A. Rammaha}
\address{Department of Mathematics, University of Nebraska--Lincoln, Lincoln, NE  68588-0130, USA} \email{mrammaha1@unl.edu}

\title[a structure-acoustics interaction model]{Global solutions to a structure acoustic interaction model with nonlinear sources.}

\date{\today}

\begin{abstract} This article focuses on a structural acoustic interaction system consisting of a semilinear wave equation defined on a smooth bounded domain $\Omega\subset\R^3$ which is strongly coupled with a 
Berger plate equation acting only on a flat part of the boundary of $\Omega$.  In particular, the source terms acting on  the wave  and plate equations are allowed to have arbitrary growth order. We  employ a standard Galerkin approximation scheme to establish a rigorous proof of the existence of  local weak solutions. In addition, under some conditions on the parameters in the system, we prove such solutions exist globally in time and depend continuously on the initial data.
\end{abstract}

\maketitle

\section{Introduction}\label{S1}

\subsection{The Model}

\noindent{}Let $\O\subset\R^3$ be a bounded, open, connected domain with smooth boundary $\partial\O=\overline{\G_0\cup\G}$, where $\G_0$ and $\G$ are two disjoint, open, connected sets of positive Lebesgue  measure.  Moreover, $\G$ is a \emph{flat} portion of the boundary of $\O$ and is referred to as the elastic wall, whose dynamics are described by the Berger plate or beam equation. We refer the reader to \cite{Chueshov-1999} and the references quoted therein  for more details on the Berger model. The acoustic medium in the chamber $\O$ is described by a semilinear wave equation influenced by a restoring source.  The resulting relationship is represented in the following coupled PDE system:
\begin{align}\label{PDE}
\begin{cases}
u_{tt}-\Delta u +|u|^{p-1}u= 0 &\text{ in } \O \times (0,T), \vspace{.05in} \\
w_{tt}+\Delta^2w+w_t+u_t|_{\G}=h(w)&\text{ in }\G\times(0,T),  \vspace{.05in} \\
u=0&\text{ on }\G_0\times(0,T),  \vspace{.05in}\\
\partial_\nu u=w_t&\text{ on }\G\times(0,T),  \vspace{.05in} \\
w=\partial_{\nu_\G}w=0&\text{ on }\partial\G\times(0,T),  \vspace{.05in} \\
(u(0),u_t(0))=(u_0,u_1),\hspace{5mm}(w(0),w_t(0))=(w_0,w_1),
\end{cases}
\end{align}
where the initial data reside in the finite energy space, i.e., $$u_0\in H^1_{\G_0}(\O)\cap L^{p+1}(\O), \,\, u_1\in L^2(\O), \text{ and }(w_0,w_1)\in H^2_0(\G)\times L^2(\G).$$ 
The term  $|u|^{p-1}u$ represents an internal restoring source acting on the acoustic medium chamber $\O$ and is allowed to have an arbitrary power $p\geq1$.  The term $w_t$  represents a frictional internal damping on the plate, whereas $h(w)$ is an internal source on the plate that is allowed to have a bad sign which  may cause instability (blow up) in a finite time. In addition, $\nu$ and $\nu_\G$ denote the outer normal vectors to $\G$ and $\partial\G$; respectively. The part $\G_0$ of the boundary $\partial\O$ describes a rigid  wall, while  the coupling takes place on the flexible wall  $\G$.

\subsection{Literature Overview}
Structural acoustic interaction models have rich and extensive history. These models are well known in both the physical and mathematical literature and go back to the canonical models considered in \cite{Beale76,Howe1998}. In the context of stabilization and controllability of structural acoustic models there is a very large body of literature. We refer the reader to the monograph  by Lasiecka \cite{Las2002} which provides a  comprehensive overview and quotes many works on these topics. Other contributions worthy of mention include \cite{Avalos2,Avalos1,Avalos3,Avalos4,Cagnol1,MG1,MG3,LAS1999}. For instance,  questions of exact controllability or uniform stability are considered in \cite{Avalos4} for the interaction of wave/Kirchhoff plates, \cite{Cagnol1} for the interaction of  wave/shell models, and  \cite{MG1} for the interaction of 
wave/Reissner-Mindlin plates. For the case that corresponds to nonlinear aeroelastic plate problem in a flow of gas, we mention the papers \cite{Boutet-Chueshov1,Bute-Chueshov2,Chueshov3} which consider the coupled system of a linear wave equation in the upper-half space in $\R^3$ and von Karman equations on the flexible wall. 

Other central questions  include  the existence of global attractors and the analysis of their properties. This particular topic attracted considerable interest in the last three decades or so. 
In general,  structural acoustic models present several technical difficulties in proving existence of attractors, or asserting their regularity and their finite dimensionality in the presence of nonlinear damping. These challenges are an intrinsic character for the hyperbolic-like dynamics involved in studying the long time behavior of structural acoustic models.   In the presence of linear damping, there are several interesting results on the existence of global attractors \cite{Babin1992,Chueshov-1999,Hale1988,TE1}.
However, the presence of nonlinear damping  has been recognized in the literature as a source of many technical difficulties. Over the years, there has been some novel progress in this area, particularly for wave equations influenced by nonlinear damping \cite{Feireisl1,F,LasieckaRuzmaikina2002,Prazak2002}. For structural acoustic models and other related models we mention the work of Bucci et al \cite{BCL} and the work by Chueshov and Lasiecka and others  \cite{Chueshov-2,CL2,ChLa1,CLT1,CLT2}.
In particular, \cite{ChLa1} provides a comprehensive account of new abstract results, along with the analysis of relevant PDE examples such as wave and plate equations with nonlinear damping and critical nonlinear source terms.

Nonlinear wave equations under the influence of damping and sources has been attracting considerable attention in the research field of analysis of nonlinear PDEs. We briefly give an overview of some related results in the literature regarding wave equations and systems of wave equations. In \cite{GT}, Georgiev and Todorova considered a semilinear wave equation with frictional damping and a subcritical source term. The paper \cite{GT} provided the local and global solvability of the equation, and also provided a blow up result which ignited considerable interest in the area.
Consequent results on wave equations with subcritical sources were established in \cite{AR2,CCL,PR,RS2,V3}. We also would like to mention the works \cite{BLR2,BLR1,BLR3} on wave equations influenced by \emph{degenerate} damping and source terms. Well-posedness results for wave equations with supercritical sources include the breakthrough  papers by Bociu and Lasiecka \cite{BL2,BL1} and the papers on systems of wave equations \cite{GR1,GR2,GR}. For other related results on wave equations  involving supercritical sources we mention 
\cite{GRSTT,MOHNICK1,RammahaKass2,PRT-1,PRT-p-Laplacain} and the references therein.

In this manuscript, we follow a similar approach by Lions \cite{Lions1969} to establish the existence of local weak solutions. For the case of a critical source acting on the wave equation, we prove such solutions  depend continuously on the initial data, and so these solutions are unique in the finite energy space.

\subsection{Notation}

\noindent{}Throughout the paper the following notational conventions for $L^p$ space norms and inner products will be used, respectively:
\begin{align*}
&||u||_p=||u||_{L^p(\O)}, &&(u,v)_\O = (u,v)_{L^2(\O)},\\
&|u|_p=||u||_{L^p(\G)},&&(u,v)_\G = (u,v)_{L^2(\G)}.
\end{align*}
We also use the notation  $\g u$ to denote the \emph{trace} of $u$ on $\G$ and we write $\frac{d}{dt}(\g u(t))$ as $\g u_t$ or $\g u'$. Occasionally, we also use the notation $u_{|_\G}$ to mean $\g u$. We also use at times the notation $u'$ to mean $u_t$. 
As is customary, $C$ shall always denote a positive constant which may change from line to line.

\noindent{}Further, we put 
$$\H(\O)=\{u\in H^1(\O):u|_{\G_0}=0\}.$$ 
It is well-known that the standard norm
$\norm{u}_{\H(\O)}$ is equivalent to $\norm{\grad u}_2$. Thus, we put:
$$\norm{u}_{\H(\O)}= \norm{\grad u}_2.$$
For a similar reason, we put:
$$\norm{w}_{H_0^2(\G)}= \abs{\Delta w}_2.$$
\noindent{} Relevant to this work in the entire paper, we define the Banach space $X$ and its norm by:
\begin{align*}
X=\H(\O)\cap L^{p+1}(\O),&&\|u\|_{X}=\|\nabla u\|_2+\|u\|_{p+1}.
\end{align*}
For a Banach space $Y$, we denote the duality pairing between the dual space $Y'$ and $Y$ by 
$\<\cdot,\cdot\>_{Y',Y}$. That is, 
\begin{align*}
\< \psi,y \>_{Y',Y} =\psi(y)\text{ for }y\in Y,\, \psi \in Y'.
\end{align*}

\noindent{}Throughout the paper, the following Sobolev imbeddings will be used often without mention:
\begin{align*}\begin{cases}
H^{1-\epsilon}(\O)\hookrightarrow L^{\frac{6}{1+2\epsilon}}(\O)\text{ for }\epsilon\in[0,1],  \vspace{.05in}\\
H^{1-\epsilon}(\O) \tinto H^{\frac{1}{2}-\epsilon}(\G)\hookrightarrow L^{\frac{4}{1+2\epsilon}}(\G)\text{ for }
\epsilon\in [0,\frac{1}{2}],  \vspace{.05in} \\
H^1(\Gamma)\hookrightarrow L^q(\G) \text{ for all } 1\leq q<\infty.
\end{cases}\end{align*}
 As it occurs so frequently we shall pass to subsequences consistently without re-indexing.

\subsection{Main Results}
Throughout this paper, we study (\ref{PDE}) under the following  general  assumptions:
\begin{assumption}\label{ass}
We assume that the sources in (\ref{PDE}) are $\R$-valued functions satisfying:
\begin{itemize}
\item $1\leq p<\infty$, \vspace{.1in}
\item $h\in C^1(\R)\text{ such that }|h'(u)|\leq C(|u|^{q-1}+1)\text{ with }1\leq q<\infty$.
\end{itemize}
\end{assumption}

\begin{rmk}\label{rem1}
 As the following bounds will be used often throughout the paper it is worthy of note that the above assumption  implies that
	\begin{align*}
	\begin{cases}
	 \Big| |u|^{p-1}u-|v|^{p-1}v\Big|\leq C(|u|^{p-1}+|v|^{p-1})|u-v|, \vspace{.1in}\\
	|h(u)|\leq C(|u|^q +1),\quad|h(u)-h(v)|\leq C(|u|^{q-1}+|v|^{q-1}+1)|u-v|.	
	\end{cases}	
	\end{align*}	
\end{rmk}

\noindent{} We begin by introducing the definition of a suitable weak solution for \eqref{PDE}.
\begin{defn}\label{def:weaksln}
A pair of functions $(u,w)$ is said to be a weak solution of \eqref{PDE} on the interval $[0,T]$ provided:
	\begin{enumerate}[(i)]
		\setlength{\itemsep}{5pt}
		\item\label{def-u} $u\in C_w([0,T]; X)$,  $u_t\in C_w([0,T];L^2(\O))$,
		\item\label{def-v} $w\in C_w([0,T];H^2_0(\G))$,  $w_t\in C_w([0,T];L^2(\G))$,
		\item\label{def-uic} $(u(0),u_t(0))=(u_0,u_1) \in H^1_{\G_0}(\O)\times L^2(\O)$, 		                         
		\item\label{def-wic} $(w(0),w_t(0))=(w_0,w_1) \in H^2_0(\G)\times L^2(\G)$,
		\item\label{def-ws} The functions $u$ and $v$ satisfy the following variational  identities 
		for all $t\in[0,T]$: 
\begin{align}\label{wkslnwave}
(u_{t}(t),\phi(t))_\O &- (u_1,\phi(0))_\O-\int_0^t ( u_t(\tau), \phi_t(\tau) )_\O d\tau
+\int_0^t (\nabla u(\tau), \nabla\phi(\tau) )_\O d\tau \notag \\
&-\int_0^t  (w_t(\tau), \g\phi(\tau) )_\G d\tau+\int_0^t\int_\Omega |u(\tau)|^{p-1}u(\tau)\phi(\tau) dxd\tau=0,
\end{align}
\begin{align}\label{wkslnplt}
(w_t(t) & + \g u(t),\psi(t) )_\G  -(w_1 +\g u(0) ,\psi(0))_\G -\int_0^t (w_t(\tau), \psi_t(\tau) )_\G d\tau \notag \\
& -\int_0^t (\g u(\tau), \psi_t(\tau) )_\G d\tau+\int_0^t (\Delta w(\tau), \Delta\psi(\tau) )_\G d\tau \notag \\
&+\int_0^t ( w_t(\tau), \psi(\tau) )_\G d\tau=\int_0^t\int_{\G}h(w(\tau))\psi(\tau) d\G d\tau,
\end{align}
for all test functions $\phi\in C_w([0,T];X)$ with $\phi_t\in L^2(0,T;L^2(\O))$, and $\psi\in C_w\left([0,T];H^2_0(\G)\right)$ with $\psi_t\in L^{2}(0,T;L^2(\G))$.
	\end{enumerate}
\end{defn}

\begin{rmk}
In Definition~\ref{def:weaksln} above, $C_w([0,T]; X)$ denotes the space of weakly continuous (often called scalarly continuous) functions from $[0,T]$ into a Banach space $X$.  That is, for each $u\in C_w([0,T];X)$ and  $f \in X'$ the map $t\mapsto \langle f, u(t) \rangle_{X',X}$ is continuous on $[0,T]$.
\end{rmk}

\noindent{}Our principal result is the existence of local solutions of problem \eqref{PDE} in the following sense.
\begin{thm}\label{thm:exist}
	Under the validity of Assumption \ref{ass}, problem \eqref{PDE} possesses a local weak solution, $(u,w)$, in the sense of Definition~ \ref{def:weaksln} on a non-degenerate interval  $[0,T]$, where $T$ depends upon the  initial positive energy $\E(0)$ (where $\E(t)$ is defined below).  Furthermore, if in addition $1\leq p\leq 3$, then the said solution $(u,w)$ satisfies the following energy identity for all $t\in [0,T]$:
\begin{align}\label{energyeq}
&\mathscr{E}(t)+\int_0^t |w_t(\t)|_2^{2} \, d\tau=\mathscr{E}(0)+\int_0^t\int_{\G}h(w)w_td\G d\tau,
\end{align}
where
\begin{align}\label{energy}
\mathscr{E}(t)=\frac{1}{2}\left(\|u_t(t)\|_2^2+\|\nabla u(t)\|_2^2+|w_t(t)|_2^2+|\Delta w(t)|_2^2\right)+\frac{1}{p+1}\|u(t)\|_{p+1}^{p+1}.
\end{align}
\noindent{}If  $p>3$, then the solution $(u,w)$ satisfies the energy inequality:
\begin{align}\label{energyineq1}
\mathscr{E}(t)+ \int_0^t |w_t(\t)|_2^{2} \, d\tau \leq \mathscr{E}(0)+\int_0^t\int_{\G}h(w)w_td\G d\tau \text{\,  a.e. } [0,T].
\end{align}
\noindent{}Equivalently, \eqref{energyineq1} can also be written as  
\begin{align}\label{energyineq2}
E(t)+ \int_0^t |w_t(\t)|_2^{2} \, d\tau\leq E(0) \text{\,  a.e. } [0,T],
\end{align}
with $E(t)=\mathscr{E}(t)-\int_\G H(w(t))d\G$, where $H$ is the primitive of $h$, i.e., $H(w)=\int_0^w h(s)ds$.
\end{thm}

Although the source term acting on the plate equation can have a ``bad" sign which may cause blow up in finite time, our next result states that solutions established  by Theorem \ref{thm:exist} are indeed global solutions, provided the plate source term is essentially linear. 
\begin{thm}\label{thm:global}
In addition to Assumption \ref{ass}, assume $q=1$.  Then any solution $(u,w)$ furnished by Theorem~\ref{thm:exist} is a global weak solution and the existence time $T$ may be taken arbitrarily large.
\end{thm}

\begin{thm}\label{thm:CDID}
In addition Assumption \ref{ass}, assume $p\leq 3$ and $U_0=(u_0,w_0,u_1,w_1)\in H$ is an initial data with a corresponding weak solution $(u,w)$ of \eqref{PDE}, where 
$ H=H^1_{\G_0}(\O)\times H^2_0(\G)\times L^2(\O)\times L^2(\G)$.  If $U_0^n=(u_0^n,w_0^n,u_1^n,w_1^n)$ is a sequence of initial data such that  $U^n_0\lra U_0$ in $H$,  as $n\lra \infty$, then the corresponding weak solutions $(u^n,w^n)$ with initial data $U_0^n$ satisfy:   
$$(u^n,w^n, u_t^n,w_t^n)\lra (u,w, u_t,w_t) \text{ in }L^\infty(0,T;H), \text{ as } n\lra \infty,$$
where $0<T<\infty$ is chosen to be independent of $n\in \N$.
\end{thm}

\begin{cor}\label{cor:uni}
In addition to Assumptions \ref{ass}, assume $ p\leq 3$.  Then, weak solutions of   \eqref{PDE} (in the sense of Definition  \ref{def:weaksln}) are unique.
\end{cor}

The paper is organized as follows. Sections \ref{S2} and \ref{S3} are devoted to the proof of Theorem \ref{thm:exist}. 
In Sections \ref{S4} and  \ref{S5}
 we complete the proofs of Theorems \ref{thm:global} and \ref{thm:CDID}

\section{Existence of Local Solutions}\label{S2}

\subsection{Approximate Solutions}
We begin by  selecting a sequence  $\{e_j\}_1^\infty \subset X= \H(\O)\cap L^{p+1}(\O)$ with the following properties:
\begin{align}\label{basis1}
\begin{cases}
e_1,\cdots, e_N \text{  are linearly independent for every } N \in \N, \text{and} \vspace{.1in}\\
\text{The set of all finite linear combinations of the form:}\\
 \Big\{ \sum_{j=1}^N c_j e_j: \, c_j\in \R,\, N\in \N  \Big\} \text{ is dense in } X.
\end{cases}
\end{align}

\noindent{}Let $B=\Delta^2$ with its domain $\mathscr{D}(B)=H^4(\G)\cap H^2_0(\G)$.  It is well known that $B$ is positive, self-adjoint, and $B$ is the inverse of a compact operator.  Moreover, $B$ has the infinite sequence of positive eigenvalues $\{\mu_n:n\in\mathbb{N}\}$ and a corresponding sequence of eigenfunctions $\{\sigma_n:n\in\mathbb{N}\}$ which can be normalized to form an orthonormal basis for $H^2_0(\G)$ while remaining an orthogonal basis for $L^2(\G)$.  In particular it is well known that the standard inner product $(w,z)_{H^2_0(\G)}$ is equivalent to $(\Delta w,\Delta z)_\G$, and in turn $|\Delta w|_2$ is equivalent to the standard norm on $H^2_0(\G)$. Thus, we put: 
\begin{align}\label{2.2}
(w,z)_{H^2_0(\G)}=(\Delta w,\Delta z)_\G, \quad \norm{w}_{H^2_0(\G)} = |\Delta w|_2.
\end{align}

For given initial data $(u_0,u_1)\in H^1_{\Gamma_0}(\Omega)\times L^2(\Omega)$ we can find for each $N\in\N$ sequences of real numbers $\{u^0_{N,j}\}_{N,j=1}^\infty$, $\{u^1_{N,j}\}_{N,j=1}^\infty$ such that
\begin{align}\label{1.6}
\begin{cases}
\sum_{j=1}^N u^0_{N,j} e_j \rightarrow u_0\text{ strongly in }X,  \text{ as  } N \rightarrow \infty,\vspace{.1in}\\
\sum_{j=1}^N u^1_{N,j} e_j \rightarrow u_1\text{ strongly in }L^2(\Omega),  \text{ as  } N \rightarrow \infty.
\end{cases}
\end{align}
Similarly, for given initial data $(w_0,w_1)\in H^2_0(\G)\times L^2(\G)$,
we may find sequences of scalars $\{w^0_{j}=(\Delta w_0,\Delta\sigma_j)_\G: \, j\in \N\}$ and $\{w^1_j=\frac{1}{|\sigma_j|_2}(w_1,\sigma_j)_\G:  \, j\in \N\}$ such that 
\begin{align}\label{1.7}
\begin{cases}
\sum_{j=1}^N w^0_{j} \sigma_j \rightarrow w_0\text{ strongly in } H^2_0(\G)  \text{ as  } N \rightarrow \infty, \vspace{.1in}\\
\sum_{j=1}^N w^1_j \sigma_j\rightarrow w_1\text{ strongly in }L^2(\G),  \text{ as  } N \rightarrow \infty.
\end{cases}
\end{align}

\noindent{}We now seek to construct a sequence of approximate solutions in the form 
\begin{align}\label{approxform}
\begin{cases}
u_N(x,t)=\sum_{j=1}^N u_{N,j}(t)e_j(x),\vspace{.1in}\\
w_N(x,t)=\sum_{j=1}^Nw_{N,j}(t)\sigma_j(x),
\end{cases}
\end{align} 
that satisfy the system of ODEs:
\begin{align}\label{approx1}
\begin{cases}
(u''_N,e_j)_\O+(\nabla u_N,\nabla e_j)_\O-(w'_N,\g e_j)_\G+\int_\O|u_N|^{p-1}u_Ne_jdx=0, \vspace{.1in}\\
(w''_N,\sigma_j)_\G+(\Delta w_N,\Delta \sigma_j)_\G+(w'_N,\sigma_j)_\G+(\g u'_N,\sigma_j)_\G
=\int_\G h(w_N)\sigma_jd\G,
\end{cases}
\end{align}
with initial data 
\begin{align}\label{approx2}
\begin{cases}
u_{N,j}(0)=u^0_{N,j},\,\,\,u'_{N, j}(0)=u^1_{N,j}, \vspace{.1in}\\
w_{N, j}(0)=w^0_{j},\,\,\,w'_{N, j}(0)=w^1_j.
\end{cases}
\end{align}
where  $j=1,\ldots,N$.

\noindent{}We note here that \eqref{approx1}--\eqref{approx2}  is an initial-value problem for a second order $2N\times 2N$ system of ordinary differential equations with continuous nonlinearities in the unknown functions $u_{N,j}$ and $w_{N,j}$ and their time derivatives. Therefore, it follows from the Cauchy-Peano theorem that for every $N\geq 1$, \eqref{approx1}--\eqref{approx2} has a solution $u_{N,j}$, $w_{N,j}\in C^2([0,T_N])$, $j=1,\ldots N$, for some $T_N>0$.

\subsection{A priori estimates} We aim to demonstrate that each of the approximate solutions $(u_N,w_N)$ exists on a non-degenerate interval $[0,T]$, where $T$ is independent of $N$. 

\begin{prop}\label{prop:apriori} Each approximate solution $(u_N,w_N)$ exists on a non-degenerate interval 
$[0,T]$, where  $T$ depends on the initial positive energy $\E(0)$ and other generic constants.
Further, the sequences of approximate solutions $\{u_N\}_1^\infty$ and $\{w_N\}_1^\infty$ satisfy
	\begin{subequations}\begin{align}
	\{u_N\}_1^\infty&\text{ is a bounded sequence in }L^{\infty} (0,T; X), \label{aprioriresultsu}\\
	\{u_N'\}_1^\infty&\text{ is a bounded sequence in }L^{\infty} (0,T;L^2(\O)),\label{aprioriresultsut}\\
	\{w_N\}_1^\infty&\text{ is a bounded sequence in }L^\infty(0,T;H^2_0(\G)),\label{aprioriresultsw}\\
	\{w_N'\}_1^\infty&\text{ is a bounded sequence in }L^\infty(0,T;L^2(\G)).\label{aprioriresultswt}
	\end{align}\end{subequations}
\end{prop}
\begin{proof}
	Multiplying the first equation of \eqref{approx1} by $u'_{N,j}$ and summing over $j=1,\ldots,N$, we obtain 
\begin{align}\label{aprioriwave1}
\frac{1}{2}\frac {d}{dt}\Big(\|u'_N(\tau)\|_2^2 &+\|\nabla u_N(\tau)\|_2^2\Big)  -(w'_N(\tau),u'_N(\tau))_\G \notag \\ &+\int_\Omega |u_N(\tau)|^{p-1}u_N(\tau)u'_N(\tau)dx=0,
\end{align}
for each $\tau\in[0,T_N]$. 
Similarly, multiplying the second equation of \eqref{approx1} by $w'_{N,j}$ and summing over $j=1,...,N$, one has
\begin{align}\label{aprioriplate1}
\frac{1}{2}\frac {d}{dt}\Big(|w'_N(\tau)|_2^2+|\Delta w_N(\tau)|_2^2\Big)& +(u'_N(\tau),w'_N(\tau))_\G +|w'_N(\tau)|^{2}_{2} \notag \\ &=\int_{\Gamma} h(w_N(\tau))w'_N(\tau)d\Gamma,
\end{align}
for each $\tau\in[0,T_N]$. 

By adding \eqref{aprioriwave1} and \eqref{aprioriplate1} and integrating with respect to $\t$ over $[0,t]$,  we obtain
\begin{align}\label{apriori2}
&\mathscr{E}_N(t)+\int_0^t|w'_N(\tau)|_{2}^{2}d\tau=\mathscr{E}_N(0)+\int_0^t\int_{\Gamma}h(w_N(\tau))w'_N(\tau)d\Gamma d\tau,
\end{align}
where $\mathscr{E}_N(t)$ 
is  the positive energy of the system given by:
\begin{align}
\mathscr{E}_N(t)=&\frac{1}{2}\left(\|u'_N(t)\|_2^2+\|\nabla u_N(t)\|_2^2+|w'_N(t)|_2^2 
+|\Delta w_N(t)|_2^2\right)  \notag\\    +& \frac{1}{p+1}\|u_N(t)\|_{p+1}^{p+1}.
\end{align}

Let us note here that due to the strong convergence in \eqref{1.6} and \eqref{1.7}, $\mathscr{E}_N(0)\leq C$ for some positive  constant  $C$ independent of $N$, but depends upon $\E(0)$.  In order to produce a suitable bound on $\mathscr{E}_N(t)$ we shall  estimate the term involving $h(w_N)$ as follows. By the assumption imposed on $h$, we have
\begin{align}\label{apriori-2.2}
\left|\int_{\Gamma}h(w_N(\tau))w_N(\tau)'d\Gamma\right|&\leq C\Big||w_N(\tau)|^{q}+1\Big|_2|w'_N(\tau)|_2 \notag \\
& \leq C\left(|w_N|_{2q}^{2q}+|w'_N|_2^2 +1\right)\notag \\
&\leq C_1\left(|\Delta w_N|_2^{2q}+|w'_N|_2^2 +1\right),
\end{align}
where we have used H\"older's and Young's inequalities, and the positive constant $C_1$  in (\ref{apriori-2.2}) is independent of $ N$. 

Combining \eqref{apriori2} and \eqref{apriori-2.2} yields:
\begin{align}\label{apriori-3}
\mathscr{E}_N(t)+\int_0^t|w'_N(\tau)|_{2}^{2}d\tau&\leq C+C_1\int_0^t\left(|\Delta w_N(\tau)|_{2}^{2q}+|w'_N(\tau)|_{2}^2+1\right)d\tau  \notag\\
&\leq C+C_1\int_0^t\left(\mathscr{E}_N(\tau)+1\right)^qd\tau.
\end{align}
By putting $y_N(t)=1+\mathscr{E}_N(t)$, then \eqref{apriori-3} yields
\begin{align}\label{apriori-4}
y_N(t)\leq C+C_1\int_0^t y_N(\tau)^q d\tau.
\end{align}
If $q=1$, then it follows by Gronwall's inequality that $y_N(t) \leq Ce^{C_1t}$, for all $t\geq 0$ and all $N\in\N$.
However, if $q>1$, then by using a standard comparison theorem, \eqref{apriori-4} yields that $y_N(t)\leq z(t)$, where $z(t)=\left(C^{1-q}-C_1(q-1)t\right)^{\frac{-1}{q-1}}$ is the solution of the Volterra integral equation
\begin{align}\label{apriori-5}
z(t)=C+C_1\int_0^tz(\tau)^qd\tau.
\end{align}
Although $z(t)$ blows up in finite time, nonetheless, there exists a time $0<T<T_N$ depending on $q$ and 
$ \E(0)$ such that $y_N(t)\leq z(t)\leq C_0$ for all $t\in [0,T]$, where $C_0$ is independent of $N$, but depending on  $q$ and $ \E(0)$.  Hence, for all $N\geq 1$ and any $q\geq 1$, one has $y_N(t)\leq C_0$ for all $t\in[0,T]$, establishing the proposition.
\end{proof}

\noindent{}An immediate consequence of Proposition~\ref{prop:apriori} along with the Banach-Alaoglu theorem and the well-known Aubin-Lions-Simon Compactness Theorem (e.g., \cite[Thm. II.5.16]{Boyer2013}) is the following:
 \begin{subequations}
\begin{cor}\label{cor:converg} For all sufficiently small $\epsilon>0$ there exists a function $u$ and a subsequence of $\{u_N\}$ (still denoted by $\{u_N\}$) such that  
\begin{alignat}{2}
	 	&u_N\rightarrow u\text{ weak}^*\text{ in }L^\infty(0,T;X), \label{converg:a} \\
		&u'_N\rightarrow u'\text{ weak}^*\text{ in }L^\infty(0,T;L^2(\Omega)),\label{converg:d} \\
		&w_N\rightarrow w\text{ weak}^*\text{ in }L^\infty(0,T;H^2_0(\Gamma)),\label{converg:c} \\
		&w'_N\rightarrow w'\text{ weak}^*\text{ in }L^\infty(0,T;L^2(\Gamma)),\label{converg:e}\\
		&u_N\rightarrow u \text{ strongly in }C([0,T];H^{1-\epsilon}(\Omega)),\label{converg:f}\\
		&w_N\rightarrow w\text{ strongly in } C([0,T];H^1_0(\Gamma)),\label{converg:g}\\
		&\g u_N\rightarrow \g u \text{ strongly in }  C([0,T]; L^{\frac{4}{1+2\epsilon}}(\G)).\label{converg:trace}
\end{alignat}
for all $\e \in (0, \frac{1}{2}]$.
\end{cor}
\end{subequations}
\smallskip
\subsection{Passage to the limit and verification of  (\ref{wkslnwave}) }\label{S2-lim}  
\noindent{}We begin by considering the wave portion of \eqref{approx1}, and after integrating over $[0,t]$, we obtain:

\begin{align}\label{WavePassage}
(u'_N(t),e_j)_\O-(u'_N(0),e_j)_\O & +\int_0^t(\nabla u_N(\tau),\nabla e_j)_\O d\tau -\int_0^t(w'_N(\tau),\gamma e_j)_\G d\tau \notag \\
& +\int_0^t \int_\O|u_N(\tau)|^{p-1}u_N(\tau)e_jdx d\tau=0,
\end{align}
where $j=1,...,N$. 

We first note that \eqref{converg:d} implies that
\begin{align}\label{Wave-1}
( u_N'(t), e_j)_\O\longrightarrow(u'(t), e_j)_\O\text{ weak}^*\text{ in }L^\infty(0,T).
\end{align}
Also, from \eqref{converg:a} we see 
$$u_N\longrightarrow u \text{ weak}^* \text{  in } 
L^{\infty}(0,T;H_{\G_0}^1(\O))=\left(L^1(0,T;(H_{\G_0}^1(\O))'\right)',$$ 
and as a result we conclude that:
\begin{align}\label{WaveGrad}
(\nabla u_N(\t),\nabla e_j)_\O\longrightarrow(\nabla u(\t),\nabla e_j)_\O\text{ weak}^*\text{ in }L^\infty(0,T).
\end{align}
Since $e_j\in X$ and by the continuity of the trace map $H_{\G_0}^1(\O)\tinto  L^4(\G) $, then it follows from \eqref{converg:e} that
\begin{equation}\label{WaveTrace}
(w'_N(\t),\gamma e_j)_\G\longrightarrow(w'(\t),\gamma e_j)_\G\text{ weak}^*\text{ in }L^\infty(0,T).
\end{equation}

\begin{prop}\label{prop:wavesource}On a subsequence, which is still labeled as $\{u_N\}_1^\infty$, we have:
\begin{align} \label{wavesource}
|u_N|^{p-1}u_N\lra |u|^{p-1}u \text{ weakly}  \text{  in } L^{\frac{p+1}{p}}(\O \times (0,T)).
\end{align}
\begin{proof}
\noindent{}By invoking (\ref{converg:f}), then there is a subsequence, labeled as $\{u_N\}_{N=1}^\infty$, such that  $u_N\longrightarrow u$ pointwise a.e. in $\, \Omega\times(0,T)$, which implies that 
  $|u_N|^{p-1}u_N \rightarrow |u|^{p-1}u$ pointwise a.e. in $\, \Omega\times(0,T)$. Since the sequence
$\{u_N\}_{N=1}^\infty$ is bounded $L^\infty(0,T;L^{p+1}(\O))$ from Proposition~\ref{prop:apriori}, and so  
$\{|u_N|^{p-1}u_N\}_{N=1}^\infty$ is bounded in $L^{\frac{p+1}{p}}(\O \times (0,T))$.  Then,  (\ref{wavesource}) follows immediately from a standard result in analysis. 
\end{proof}
\end{prop}
\begin{rmk}
Proposition~\ref{prop:wavesource} easily implies the following convergence:
\begin{align}\label{wavesource-1}
\int_0^t \int_\O|u_N(\t)|^{p-1}u_N(\t)e_jdx d\t\longrightarrow \int_0^t \int_\O|u(\t)|^{p-1}u(\t)e_j dx d\t, \text{ for } t\in[0,T].
\end{align}
\end{rmk}

\noindent{} By noting that $\chi_{[0,t]} \in L^1(0,T)$ for $t\in [0,T]$, and recalling the strong convergence of 
$u'_N(0)$ in \eqref{1.6}, then by combining (\ref{Wave-1})-(\ref{wavesource-1}), we are justified in passing to the limit in (\ref{WavePassage}) to obtain:
\begin{align}\label{Wavepostpassage}
(u'(t),e_j)_\O&- (u_1,e_j)_\O+\int_0^t(\nabla u(\tau),\nabla e_j)_\O d\tau-\int_0^t(w'(\tau),\gamma e_j)_\G d\tau \notag\\
&+\int_0^t\int_\O|u(\tau)|^{p-1}u(\tau)e_jdxd\tau=0,
\end{align}
where (\ref{Wavepostpassage}) is valid for all $j\in\N$ and a.e. $t\in[0,T]$. 

\noindent{}Now, for any $\phi\in X$, there exists a sequence $\phi_k=\sum_{j=1}^k a_{k,j}e_j$ which converges to $\phi$ strongly in $X$.  By linearity, one can replace  $e_j$ in \eqref{Wavepostpassage} with $\phi_k$, and then pass to the limit as $k \rightarrow \infty$ to obtain:
\begin{align}\label{Wavepostpassagephi}
(u'(t),\phi)_\O& - (u_1,\p)_\O+\int_0^t(\nabla u(\tau),\nabla \phi)_\O d\tau-\int_0^t(w'(\tau),\gamma \phi)_\G d\tau \notag \\
&+\int_0^t\int_\O|u(\tau)|^{p-1}u(\tau)\phi dxd\tau=0,
\end{align}
for all $\p \in X$ and a.e. $t\in[0,T]$.

Before proceeding further, we pause to verify that $u''$ has the desired additional regularity.
\begin{lem}\label{lem:utt}
The limit function $u$ identified in Corollary~\eqref{cor:converg} verifying identity \eqref{Wavepostpassagephi} satisfies $u''\in L^\infty(0,T; X')$.
\end{lem}
\begin{proof}
Let us first note the inclusions $X \subset L^2(\O) \subset X' $, where
the injections are continuous with dense ranges. In addition, 
$$\<f, \p\>_{X',X} =(f, \p)_\O, \text{ for all } f \in L^2(\O) \text{ and all } 
  \p \in X.$$
Thus, given any $\phi\in X$ we obtain from \eqref{Wavepostpassagephi} that
\begin{align}\label{2.24}
\<u'(t), \p\>_{X',X} &=(u'(t),\phi)_\O=  (u_1,\p)_\O-\int_0^t(\nabla u(\tau),\nabla \phi)_\O d\tau  \notag \\ &+\int_0^t(w'(\tau),\gamma \phi)_\G d\tau
-\int_0^t\int_\O|u(\tau)|^{p-1}u(\tau)\phi dxd\tau,
\end{align}
wherein it is clear from \eqref{2.24}  that $\<u'(t), \p\>_{X',X}$ coincides with an absolutely continuous function on 
$[0,T]$ with 
\begin{align}\label{2.24a}
\< u''(t), \p\>_{X',X} =\frac{d}{dt} (u'(t),\phi)_\O=  & -(\nabla u(t), \nabla \phi)_\O+ (w'(t),\gamma \phi)_\G \notag \\ &
-\int_\O |u(t)|^{p-1}u(t) \phi dx.
\end{align}
By employing  H\"older's inequality and the Sobolev Imbedding Theorem, we obtain
\begin{align}\label{2.26}
\left| \<u''(t), \p\>_{X',X}\right| &  \leq \abs{ (\nabla u(t),\nabla \phi)_\O} + 
\abs{(w'(t),\gamma \phi)_\G}  + \int_\O|u(t)|^{p} |\phi| dx \notag  \\
 & \leq \|\nabla u(t)\|_2\|\nabla\phi\|_2+|w'(t)|_2|\gamma\phi|_2
+\|u(t)\|_{p+1}^p \|\phi\|_{p+1} \vspace{.1in}  \notag  \\
& \leq C\Big(\|\nabla u(t)\|_2+|w'(t)|_2+\|u(t)\|_{p+1}^p\Big)\|\phi\|_{X}.
\end{align}
By the regularity enjoyed by $u$ and $w$ as stated  in Corollary \ref{cor:converg}, we conclude that
$u''\in L^\infty(0,T; X')$. 
\end{proof}
%%%%%%%%%%%%%%%%%%%%%%

\subsection{Proper verification of (\ref{wkslnwave})  }
We now must show that the limit function $u$ satisfies the variational identity (\ref{wkslnwave}) which permits time dependent test functions.  By a density arguemnt as in \cite[Prop. A.1]{PRT-p-Laplacain} it can be shown that the regularity afforded by Lemma~\ref{lem:utt} implies the following: 
For any test function  $\phi\in C_w([0,T];X)$ with $\phi_t\in L^2(0,T;L^2(\O))$, the function $(u'(t),\phi(t))_\O$ coincides with an absolutely continuous function on $[0,T]$ and one has the following product rule in the distributional  sense:
\begin{align}\label{WavePR}
\frac{d}{dt}(u'(t),\phi(t))_\O&=\<u''(t),\phi(t)\>_{X',X}+(u'(t),\phi'(t))_\O.
\end{align} 
With this at hand and noting that the function $\p$ in  \eqref{Wavepostpassagephi}  is time independent, we may express \eqref{Wavepostpassagephi} equivalently as 
\begin{align}\label{2.28}
\int_0^t  \<u''(\t),\phi\>_{X',X}d\t&+\int_0^t(\nabla u(\tau),\nabla \phi)_\O d\tau-\int_0^t(w'(\tau),\gamma \phi)_\G d\tau \notag \\
&+\int_0^t\int_\O|u(\tau)|^{p-1}u(\tau)\phi dxd\tau=0,
\end{align}
for all $\p \in X$. 

As each term in \eqref{2.28} is absolutely continuous we may differentiate in time and then replace $\p$ with $\phi(\t)$ where the time dependent test function $\phi(\t)$ satisfying $\phi\in C_w([0,T];X)$ with $\phi_t\in L^2(0,T;L^2(\O))$. Integrating the resulting identity on $[0,t]$ and again utilizing the product rule \eqref{WavePR} we obtain the desired identity,  namely:
\begin{align}\label{2.29}
\overbrace{(u_t(t),\phi(t))_\O - (u_1,\phi(0))_\O - \int_0^t (u'(\tau),\phi'(\tau))_\O\,d\tau}^{\int_0^t \langle u''(\tau),\phi(\tau)\rangle_{X',X} \,d\tau}  +\int_0^t(\nabla u(\tau),\nabla \phi (\t))_\O d\tau \notag \\ 
-\int_0^t(w'(\tau),\gamma \phi  (\t) )_\G d\tau +\int_0^t\int_\O|u(\tau)|^{p-1}u(\tau)\phi  (\t) dxd\tau=0,
\end{align}
which is exactly (\ref{wkslnwave}), i.e., the limit function $u$  satisfies the variational identity  (\ref{wkslnwave}) in Definition \ref{def:weaksln}. 

%%%%%%%%%%%%%%%%%%%%%%%

\subsection{Passage to the limit and verification of  (\ref{wkslnplt})}
\noindent{} Upon integrating  the plate equation in \eqref{approx1} on $[0,t]$, we obtain:
\begin{align}\label{PlatePassage}
 (w'_N(t) &,\sigma_j)_\G  - (w'_N(0),\sigma_j)_\G + \int_0^t (w'_N(\t),\sigma_j)_\G d\t 
+(\g u_N(t),\sigma_j)_\G \notag \\ & -(\g u_N(0),\sigma_j)_\G 
+\int_0^t (\Delta w_N(\t),\Delta \sigma_j)_\G d\t=\int_0^t \int_\G h(w_N(\t))\sigma_jd\G d\t,
\end{align}
for all  $j=1,\ldots, N$.  It follows easily from \eqref{converg:c}-\eqref{converg:trace} that:
\begin{align}\label{Platelimits}
\begin{cases}
(w'_N(t),\sigma_j)_\G\longrightarrow(w'(t),\sigma_j)_\G\text{ weak}^*\text{ in }L^\infty(0,T) \vspace{.1in}\\
(\Delta w_N(\t),\Delta \sigma_j)_\G\longrightarrow(\Delta w(\t),\Delta\sigma_j)_\G\text{ weak}^*\text{ in }L^\infty(0,T)  \vspace{.1in}\\
(w_N(t),\sigma_j)_\G \longrightarrow (w(t),\sigma_j)_\G \text{ strongly in }  C([0,T]),  \vspace{.1in}\\
(\g u_N(t),\sigma_j)_\G \longrightarrow (\g u(t),\sigma_j)_\G  \text{ strongly in }  C([0,T]).
\end{cases}
\end{align}
for all $j\in \N$.

For the source term in (\ref{PlatePassage}), we show that 
\begin{align}\label{2.33}
\int_\G h(w_N(\t))\sigma_jd\G  \longrightarrow \int_\G h(w(\t))\sigma_jd\G \text{ strongly in }  C([0,T]), \text{ as } N \rightarrow \infty,
\end{align}
for all $j\in \N$. Indeed, for all $\t \in [0,T]$ we have 
\begin{align}\label{2.34}
\Big| \int_\G & h(w_N(\t))\sigma_jd\G - \int_\G h(w(\t))\sigma_jd\G \Big| \notag \\
&  \leq C\int_\G (|w_N(\t)|^{q-1}+|w(\t)|^{q-1}+1)|w_N(\t)-w(\t)| |\s_j| d\G \notag \\
&  \leq C (|w_N(\t)|_{6(q-1)}^{q-1}+|w(\t)|_{6(q-1)}^{q-1}+1) |w_N(\t)-w(\t)|_2 |\s_j|_3 \notag \\
&  \leq C \sup_{\t \in [0,T]} |\grad w_N(\t)- \grad w(\t)|_2 \rightarrow 0, \text{ as } N \rightarrow \infty, 
\end{align}
where we have used in (\ref{2.34})  H\"older's inequality, the Sobolev Imbedding Theorem, and (\ref{converg:g}). Therefore, (\ref{2.33}) follows. 

By noting that $\chi_{[0,t]} \in L^1(0,T)$ for $t\in [0,T]$,  the strong convergences in (\ref{1.6})-(\ref{1.7}), and  using convergences in  \eqref{Platelimits}- \eqref{2.33}, we can now pass to the limit as $N \ra \infty$ in (\ref{PlatePassage}) to obtain the identity:
\begin{align}\label{PlatePassage-1}
& (w'(t),\sigma_j)_\G  - (w_1,\sigma_j)_\G +\int_0^t (w'(\t),\sigma_j)_\G d\t  
+(\g u(t),\sigma_j)_\G \notag \\ & -(\g u_0,\sigma_j)_\G 
+\int_0^t (\Delta w(\t),\Delta \sigma_j)_\G d\t=\int_0^t \int_\G h(w(\t))\sigma_jd\G d\t,
\end{align}
for all $j\in \N$ and a.e. [0,T]. 

Since $\{\sigma_n:n\in\mathbb{N}\}$ is an orthonormal basis for $H^2_0(\G)$, then (\ref{PlatePassage-1}) yields:
\begin{align}\label{PlatePassage-2}
(w'(t)  + \g u(t),\eta)_\G & - (w_1+ \g u_0,\eta)_\G +\int_0^t (w'(\t),\eta )_\G d\t 
 \notag \\ &  +\int_0^t (\Delta w(\t),\Delta \eta)_\G d\t=\int_0^t \int_\G h(w(\t)) \eta d\G d\t,
\end{align}
for all $\eta\in H^2_0(\Gamma)$ and a.e. $t\in [0,T]$.

Before proceeding further, we pause briefly  to verify that $\frac{d}{dt}(w'+\g u)$ has a desired additional regularity. Namely, we have the following. 
\begin{lem}\label{lem:wtt}
The limit functions $u$ and $w$ identified in Corollary~\eqref{cor:converg} verifying identity \eqref{PlatePassage-2} satisfies $\frac{d}{dt}\Big(w'+\g u\Big) \in L^\infty(0,T; H^{-2}(\G))$.
\end{lem}
\begin{proof}
In what follows, we shall use the notation $\< \cdot, \cdot\>$ to denote  the duality pairing between $H^{-2}(\O)$ and $H_0^{2}(\O)$. We first note that  $H_0^{2}(\G) \subset L^2(\G) \subset H^{-2}(\G) $, where
the injections are continuous with dense ranges. In addition, 
$$\<f, \eta\> =(f, \eta)_\G, \text{ for all } f \in L^2(\G) \text{ and all }  \eta \in H_0^{2}(\G).$$
So, for any $\eta \in H_0^{2}(\G)$ we obtain from \eqref{PlatePassage-2} that
\begin{align}\label{2.37}
\< w'(t) + \g u(t), \eta \> & =(w'(t)  + \g u(t), \eta)_\G = (w_1 + g u_0,\eta )_\G -\int_0^t (w'(\t),\eta)_\G d\t 
 \notag \\ &   -\int_0^t (\Delta w(\t),\Delta \eta)_\G d\t+\int_0^t \int_\G h(w(\t)) \eta d\G d\t.
\end{align}
It is evident from \eqref{2.37}  that $\< w'(t) + \g u(t),\eta \>$ coincides with an absolutely continuous function on $[0,T]$ with 
\begin{align}\label{2.38}
\frac{d}{dt} (w'(t)  + \g u(t), \eta)_\G  =  -(w'(t), \eta)_\G
  - (\Delta w(t),\Delta \eta)_\G +\int_\G h(w(t)) \eta d\G,
\end{align}
for almost all $t\in [0,T]$. In particular, one has
\begin{align}\label{2.39}
\Big|\<\frac{d}{dt} ( w'(t) + \g u(t)) , \eta \>\Big| & \leq  |w'(t)|_2 | \eta|_2
  +|\Delta w(t)|_2 |\Delta \eta |_2   +C\int_\G ( |w(t)|^q +1) |\eta| d\G \notag \\ 
  & \leq C\Big( |w'(t)|_2  +|\Delta w(t)|_2  + |\Delta w(t)|_2^q +1\Big) |\Delta \eta |_2, 
\end{align}
 for all $\eta \in H_0^{2}(\G)$ and for almost all $t\in [0,T]$.
By the regularity enjoyed by  $w$ as stated  in Corollary \ref{cor:converg}, we conclude that
$\frac{d}{dt}(w'+\g u) \in L^\infty(0,T; H^{-2}(\G))$.
\end{proof}
%%%%%%%%%%%%%%%%%%%%%%

\subsection{Proper verification of (\ref{wkslnplt}) }
We now must show that the limit function $w$ satisfies the variational identity (\ref{wkslnplt}) which permits time dependent test functions.  Again, by using \cite[Prop. A.1]{PRT-p-Laplacain} it can be shown that the regularity afforded by Lemma~\ref{lem:wtt} implies the following: 
For any test function  $\psi\in C_w\left([0,T];H^2_0(\G)\right)$ with $\psi_t\in L^{2}(0,T;L^2(\G))$, the function 
$(w'(t) + \g u(t),\psi(t))_\G$ coincides with an absolutely continuous function on $[0,T]$ and one has the following product rule in the distributional  sense:
\begin{align}\label{2.40}
\frac{d}{dt} (w'(t) + \g u(t),\psi(t))_\G=\<\frac{d}{dt} (w'(t) + \g u(t) ),\psi(t)\>+(w'(t) + \g u(t),\psi'(t))_\G.
\end{align} 
With the validity of  (\ref{2.40})  and noting that the function $\eta$ in  \eqref{PlatePassage-2}  is time independent, we may express \eqref{PlatePassage-2} equivalently as 
\begin{align}\label{2.41}
\int_0^t \<\frac{d}{d\t} (w'(\t) &+ \g u(\t) ),\eta \> d\t  +\int_0^t (w'(\t),\eta )_\G d\t 
 \notag \\ &  +\int_0^t (\Delta w(\t),\Delta \eta)_\G d\t=\int_0^t \int_\G h(w(\t)) \eta d\G d\t,
\end{align}
for all $\eta \in H_0^{2}(\G)$ and all $t\in [0,T]$.

As each term in \eqref{2.41} is absolutely continuous we may differentiate in time and then replace $\eta$ with $\psi(\tau)$ where the time dependent test function $\psi(\tau)$ satisfying  $\psi\in C_w\left([0,T];H^2_0(\G)\right)$ with $\psi_t\in L^{2}(0,T;L^2(\G))$. Upon integrating the resulting identity on $[0,t]$ and again utilizing the product rule \eqref{2.40} we obtain the desired identity,  namely:
\begin{align}\label{2.42}
&\overbrace{(w_t(t)  + \g u(t),\psi(t) )_\G  -(w_1 +\g u(0) ,\psi(0))_\G 
-\int_0^t  (w_t(\tau) + \g u(\tau), \psi_t (\tau) )_\G d\t }^{\int_0^t \<\frac{d}{d\t} (w'(\t) + \g u(\t) ),\psi(t) \> d\t }
\notag \\
&+\int_0^t (\Delta w(\tau), \Delta\psi(\tau)_\G d\tau 
+\int_0^t (w_t(\tau), \psi(\tau)_\G d\tau    =\int_0^t\int_{\G}h(w(\tau))\psi(\tau) d\G d\tau,
\end{align}
which is precisely (\ref{wkslnplt}). 

%%%%%%%%%%%%%%%

\subsection{Additional regularity of solutions}
In order to complete the proof of the existence statement of Theorem \ref{thm:exist}, we need to verify that the limit functions $u$ and $w$ identified in Corollary~\ref{cor:converg} satisfy  the additional regularity as stated in  of Definition~\ref{def:weaksln}. For this purpose, we shall use  a well-known result which often attributed to Lions and Magenes, as in \cite[Lem. 8.1]{LM1}. 

\begin{prop}\label{regularity}
	Up to possible modification on a set of measure zero, the limit functions $u$ and $w$ identified in Corollary~\ref{cor:converg} satisfy:
\begin{align}\label{2.43}
\begin{cases}
u\in C_w([0,T]; X), \, u_t\in C_w([0,T];L^2(\O)),  \vspace{.1in} \\
	w\in C_w([0,T];H^2_0(\G)), \,  w_t\in C_w([0,T];L^2(\G)).
\end{cases}
\end{align}
\end{prop}
\begin{proof}
As the proofs of  both parts in (\ref{2.43}) are similar, we only present the proof of the second statement. We note here that
$H_0^2(\G) \subset L^2(\G) \subset H^{-2}(\G)$  where the injections are continuous with dense ranges,  then by  \cite[Lem. 8.1, p. 275]{LM1}
\begin{align}\label{2.44}
 L^\infty(0,T;H_0^2(\G)) \cap C_w([0,T];L^2(\G))=C_w([0,T];H_0^2(\G)).
\end{align}
Since we know $w\in L^\infty(0,T;H^2_0(\G))  \text{ and }  w_t\in L^\infty(0,T;L^2(\G))$, then  after a possible modification on a set of  measure zero, $w\in C([0,T];L^2(\G)) $. It follows from (\ref{2.44}) that 
$w\in C_w([0,T];H^2_0(\G))$. 

Also, we recall from Lemma \ref{lem:wtt} that  $\frac{d}{dt}(w'+\g u) \in L^\infty(0,T; H^{-2}(\G))$ and since $w'+\g u \in L^\infty(0,T; L^2(\G))$, then up to possible modification on a set of measure zero, we conclude that $w'+\g u \in C([0,T]; H^{-2}(\G))$. However, we know from (\ref{converg:trace}) that $\g u \in C([0,T]; L^2(\G))$, and so it must be the case that $w_t \in C([0,T]; H^{-2}(\G))$. Hence, by  a similar reasoning  as in  (\ref{2.44}) above, it follows that  $w_t\in C_w([0,T];L^2(\G))$, completing the proof.
\end{proof}

%%%%%%%%%%%%%%%%%%%%%%%%%

\section{Energy Identity and Energy  Inequality}\label{S3}
This section is devoted to derive the energy identity \eqref{energyeq} in Theorem~\ref{thm:exist} in the case $1\leq p\leq 3$. One is tempted to test  (\ref{wkslnwave}) with $u_t$ and  (\ref{wkslnplt})
with $w_t$, and carry out standard calculations to obtain
energy identity. However, this  procedure is only
 \emph{formal},   since $u_t$ and $w_t$ are not regular enough and cannot be used as test functions in  
 (\ref{wkslnwave})  and (\ref{wkslnplt}). In order to overcome
this technicality  we shall use the difference quotients $D_hu$ and
$D_hw$ and their well-known properties that appeared in \cite{KL} and later in \cite{GR,RS2,Saw}.  We remind the reader that the space  $X=\H(\O)\cap L^{p+1}(\O)$ will be replaced simply by $X=\H(\O)$, since $1\leq p\leq 3$ in this section. 

\subsection{The Difference Quotient} 

Let $Y$  be a Banach space. For 
$u\in C_w([0,T];Y)$ and $h>0$, we define its \emph{symmetric difference quotient} by:
\begin{equation} \label{3-2}
D_h u(t)=\frac{u_e(t+h)-u_e(t-h)}{2h}, 
\end{equation}
where $u_e$ denotes the extension of $u$ to $\R$ given by:
\begin{align} \label{extention}
u_e(t)=
\begin{cases}
u(0) \text{\,\,\,for\,\,\,} t\leq 0, \\
u(t) \text{\,\,\,for\,\,\,} t\in (0,T), \\
u(T) \text{\,\,\,for\,\,\,} t\geq T.
\end{cases}
\end{align}
For the reader's convenience, we review the important results of the difference quotient (see for instance  \cite{GR,KL,RS2,Saw}).

\begin{prop}[\cite{KL}] \label{prop5}
Let $u\in C_w([0,T];Y)$ where $Y$ is a Hilbert space with inner product $(\cdot, \cdot)_Y$ . Then,
\begin{equation} \label{3-3}
\lim_{h\longrightarrow 0}\int_0^T (u ,D_h u)_Y dt=\frac{1}{2}
\left(\norm{u(T)}^2_{Y}-\norm{u(0)}^2_{Y}\right).
\end{equation}
If, in addition, $u_t\in C_w([0,T];Y)$, then
\begin{equation} \label{3-4}
\int_0^T (u_t, (D_h u)_t)_{Y}dt=0, \text{   for each  } h>0,
\end{equation}
and, as $h \longrightarrow 0$,
\begin{equation} \label{3-5}
D_h u(t) \longrightarrow u_t(t) \text{\,\,\,weakly in\,\,\,} Y,
\text{\,\,\,for every\,\,\,} t\in (0,T),
\end{equation}
\begin{equation} \label{3-6}
D_h u(0) \longrightarrow \frac{1}{2}u_t(0) \text{\,\,\,and\,\,\,}
D_h u(T) \longrightarrow \frac{1}{2}u_t(T)
\text{\,\,\,weakly in\,\,\,} Y.
\end{equation}
\end{prop}

\begin{prop}[\cite{GR}] \label{prop6}
Let $Y$ and $Z$ be Banach spaces. Assume $u\in L^1([0,T];Y)$ and
$u_t \in L^1(0,T;Y)\cap L^p(0,T; Z)$, where $1 \leq p <\infty$.
Then $D_h u\in L^p(0,T;   Z)$ and $\norm{D_h u}_{L^p(0,T;Z)} \leq \norm{u_t}_{L^p(0,T; Z)}$.
Moreover, $D_h u \longrightarrow u_t$ in $L^p(0,T; Z)$, as $h\longrightarrow 0$.
\end{prop}

\smallskip
%%%%%%%%%%%%%%%%%%%
\subsection{Proof of Energy Identity} \label{sec3.2} 

Throughout the proof, we fix $t \in (0,T)$ and
let $(u,w)$ be a weak solution of the system (\ref{PDE}) on $[0,T]$ in the sense of Definition \ref{def:weaksln}.
Recall the regularity of $u$ and $w$, namely:
 $u\in C_w([0,T]; \H(\O))$,  $u_t\in C_w([0,T];L^2(\O))$, $w\in C_w([0,T];H^2_0(\G))$,  and   $w_t\in C_w([0,T];L^2(\G))$. As such, we can define the difference quotient $D_h u(\t)$ on $[0,t]$ as in (\ref{3-2}), i.e.,
$D_h u(\t)=\frac{1}{2h}[u_e(\t+h)-u_e(\t-h)]$, where $u_e(\t)$ extends $u(\t)$ from $[0,t]$ to $\R$ as in (\ref{extention}); and with a similar definition of  the difference quotient $D_h w(\t)$ on $[0,t]$. In what follows, we may abuse notation by writing $u(\t)$, $w(\t)$ in place of $u_e(\t)$, $w_e(\t)$, and in particular we remind the reader here that $u'(\t)= w'(\t)= 0$ outside the segment $[0,t]$. 

We aim to first show that $D_h u(\t)$ and $D_h w(\t)$  satisfy  the required  regularity conditions to be suitable test functions in Definition \ref{def:weaksln}.
Indeed,  since $u\in C_w([0,t];H^1_{\Gamma_0}(\Omega))$ and $w\in C_w([0,t];H^2_0(\Gamma))$, then clearly
\begin{align}\label{3.7}
D_hu\in C_w ([0,t];H^1_{\Gamma_0}(\Omega)) \text{ and } D_hw\in C_w([0,t];H^2_0(\Gamma )).
\end{align}
In addition, for $0< h<\frac{t}{2}$ we note: 
\begin{align*}
(D_h u)_t(\t)=
\begin{cases}
\frac{1}{2h}[u_t(\t+h)-u_t(\t-h)], \text{\,\,\,\,\,if\,\,\,} h<\t<t-h, \vspace{.05in} \\
-\frac{1}{2h}u_t(\t-h), \text{\,\,\,\,\,if\,\,\,} t-h<\t<t, \vspace{.05in}\\
\frac{1}{2h}u_t(\t+h), \text{\,\,\,\,\,if\,\,\,} 0<\t<h,
\end{cases}
\end{align*}
with a similar definition for $(D_h w)_t(\t)$. 

Since $u_t\in C_w([0,t];L^2(\Omega))$ and $w_t\in C_w([0,t];L^2(\G))$, then it follows that:
\begin{align}\label{3.8}
(D_hu)_t\in  L^2(0,t;L^2(\O)) \text{ and }(D_hw)_t\in  L^2(0,t;L^2(\G)).
\end{align}
Thus,  \eqref{3.7}-\eqref{3.8} show that $D_hu$ and $D_hw$ satisfy  the required  regularity conditions to be suitable test functions in Definition \ref{def:weaksln}. Therefore, by taking $\p=D_hu$ in (\ref{wkslnwave}) and $\psi=D_h w$ in (\ref{wkslnplt}), we obtain (the variable $\t$ is being suppressed within the following integrals):
\begin{align}\label{3.9}
(u_{t}(t), D_hu(t))_\O & - (u_1,D_hu(0))_\O-\int_0^t  (u_t,  (D_hu)_t)_\O d\tau  +\int_0^t (\nabla u, \nabla D_hu)_\O d\tau \notag \\
&-\int_0^t (w_t, \g D_hu )_\G d\tau +\int_0^t\int_\Omega |u|^{p-1}u D_hu dx d\tau=0,
\end{align}
\begin{align}\label{3.10}
(w_t(t) & + \g u(t), D_hw (t) )_\G  -(w_1 +\g u(0) , D_hw(0))_\G -\int_0^t (w_t, (D_hw)_t )_\G d\tau \notag \\
& -\int_0^t (\g u, (D_hw)_t )_\G d\tau+\int_0^t (\Delta w, \Delta D_hw )_\G d\tau+\int_0^t ( w_t, D_hw )_\G d\tau  \notag \\ &=\int_0^t\int_{\G} h(w)  D_hw d\G d\tau.
\end{align}

We now justify passing to the limit as $h\longrightarrow 0$ in (\ref{3.9})-(\ref{3.10}) as follows:

By using  Proposition \ref{prop6}  with $Y=Z=L^2(\Omega)$, then as $h\rightarrow 0$,
\begin{align}\label{DQ1}
\begin{cases}
D_hu\lra u_t\text{ in }L^{2}(\Omega\times(0,t)),  \vspace{.1in} \\
D_hw\lra w_t\text{ in }L^{2}(\Gamma \times(0,t)).
\end{cases}
\end{align}

Since $u, \, u_t\in C_w([0,t];L^2(\O))$ and $w, \, w_t\in C_w([0,t];L^2(\G))$, then as $h\rightarrow 0$, it follows from  (\ref{3-6}) that
\begin{align*}
D_h u(0) & \longrightarrow \frac{1}{2}u_t(0) \text{\,\,\,and\,\,\,}
D_h u(t) \longrightarrow \frac{1}{2}u_t(t)  \text{\,\,\,weakly in\,\,\,} L^2(\O),\\
D_h w(0) &  \longrightarrow \frac{1}{2}w_t(0) \text{\,\,\,and\,\,\,}
D_h w(t) \longrightarrow \frac{1}{2}w_t(t)  \text{\,\,\,weakly in\,\,\,} L^2(\G).
\end{align*}
Therefore,
\begin{align}\label{3.11}
\begin{cases}
\lim_{h\rightarrow 0} \Big((u_{t}(t), D_hu(t))_\O- (u_1,D_hu(0))_\O\Big)= \frac{1}{2} \Big(\|u_t(t)\|_2^2 - \|u_t(0)\|_2^2\Big), \vspace{.1in} \\
\lim_{h\rightarrow 0}  (w_t(t)  + \g u(t), D_hw (t) )_\G = \frac{1}{2} |w_t(t)|_2^2 + \frac{1}{2} ( \g u(t), w_t (t) )_\G,   \vspace{.1in} \\
\lim_{h\rightarrow 0} (w_1 +\g u(0) , D_hw(0))_\G  = \frac{1}{2} |w_t(0)|_2^2 + \frac{1}{2} ( \g u(0), w_t (0) )_\G. \\
\end{cases}
\end{align}
Also, by (\ref{3-4})
\begin{equation} \label{3.12}
\int^t_0 (u_t,(D_h u) _t)_{\O} d\t= \int_0^t (w_t, (D_hw)_t )_\G d\tau = 0.
\end{equation}
In addition, since $u\in C_w([0,t]; \H(\O))$ and $w\in C_w([0,t];H^2_0(\G))$, then (\ref{3-3}) yields:
\begin{align}\label{3.13}
\begin{cases}
\lim_{h\rightarrow 0}  \int_0^t (\nabla u, \nabla D_hu)_\O d\t= \frac{1}{2}\Big( \|\nabla u(t)\|_2^2-\|\nabla u(0)\|_2^2\Big), \vspace{.1in} \\
\lim_{h\rightarrow 0}  \int_0^t (\Delta w, \Delta D_hw )_\G d\tau = \frac{1}{2}\Big( |\Delta w(t)|_2^2-|\Delta w(0)|_2^2\Big).
\end{cases}
\end{align}
An immediate consequence of (\ref{DQ1}) is that
\begin{align}\label{En8}
\lim_{h\lra 0}\int_0^t ( w_t, D_hw )_\G d\tau=\int_0^t  |w_t(\tau)|_2^2 d\tau.
\end{align}

Also, since  $u\in C_w([0,T];\H(\O) )$, then  $u\in L^\infty(0,T;L^6(\O))$, by the Sobolev Imbedding Theorem. The  assumption $1\leq p\leq 3$ yields,
\[
\norm{|u(t)|^{p-1}u(t)}_2 =\norm{u(t)}_{2p}^p \leq C \norm{u}_ {L^\infty(0, T; \H(\O))} < \infty.
\]
Consequently, $|u|^{p-1} u \in  L^2(\O\times(0,t))$, and from (\ref{DQ1}) we have
\begin{align}\label{3.15}
\lim_{h\rightarrow 0} \int_0^t\int_\Omega |u|^{p-1}u D_hu dx d\tau =  \int_0^t\int_\Omega |u|^{p-1}u u_t dx d\tau. 
\end{align}
In addition, since $w\in C_w([0,T];H^2_0(\G))$, then  $w\in L^\infty(0,T;L^{2q}(\G))$ for all $1\leq q \leq \infty$. Thus, the bound imposed on $h$ in Remark \ref{rem1} implies $h(w) \in L^2(\G\times(0,T)) $. As such, 
(\ref{DQ1}) implies
\begin{align}\label{3.16}
\lim_{h\lra0}\int_0^T \int_\G h(w)D_hw d\G d\tau  =\int_0^T\int_\G h(w)w_t d\G d\tau.
\end{align}
The trouble terms $\int_0^t(\g u(\tau), D_h w_t(\tau))_\G d\tau$ and $\int_0^t (w_t(\tau), \g D_hu(\tau))_\G d\tau $
are handled as follows. For all sufficiently small $h>0$, we have 
\begin{align}\label{3.17}
\int_0^t(\g u(\tau),& D_h w_t(\tau))_\G d\tau \notag \\
&=\frac{1}{2h}\Big(\int_0^t (\g u(\tau),w_t(\tau+h))_\G d\tau-\int_0^t(\g u(\tau),w_t(\tau-h))_\G d\tau\Big) \notag \\
&=\frac{1}{2h}\Big(\int_h^{t}(\g u(\tau-h),w_t(\tau))_\G d\tau-\int_{0}^{t-h}(\g u(\tau+h),w_t(\tau))_\G d\tau\Big), 
\end{align}
where we have used a change of variables in (\ref{3.17}) and the fact that $w_t=0$ outside the interval $[0,t]$.
By rearranging the terms in (\ref{3.17}), we obtain
\begin{align}\label{3.18}
\int_0^t(\g u(\tau),& D_h w_t(\tau))_\G d\tau =-\int_0^t( \g D_h u(\tau),w_t(\tau))d\tau \notag \\
& -\frac{1}{2h}\Big(\int_0^{h}(\g u(\tau-h),w_t(\tau))_\G d\tau-\int_{t-h}^{t}(\g u(\tau+h),w_t(\tau))_\G d\tau\Big)
\end{align}
We now utilize the weak continuity of $w_t$ in the last two term in (\ref{3.18}) as follows.
\begin{align}\label{3.19}
\frac{1}{2h} \int_0^{h} & (\g u(\tau-h),w_t(\tau))_\G d\tau=\frac{1}{2h}\int_0^h(\g u(0),w_t(\tau))_\G d\tau  \notag \\
&=\frac{1}{2h}\int_0^h(\g u(0),w_t(\tau)-w_t(0))_\G d\tau+\frac{1}{2h}\int_0^h(\g u(0),w_t(0))_\G d\tau  \notag\\
& \lra \frac{1}{2}(\g u(0),w_t(0))_\G, \text{ as } h\lra 0.
\end{align}
Similarly, we have 
\begin{align}\label{3.20}
\frac{1}{2h} \int_{t-h}^{t} & (\g u(\tau+h),w_t(\tau))_\G d\tau =\frac{1}{2h} \int_{t-h}^{t} (\g u(t),w_t(\tau))_\G d\tau  \notag \\
&=\frac{1}{2h} \int_{t-h}^{t}  (\g u(t),w_t(\tau)-w_t(t))_\G d\tau+\frac{1}{2h} \int_{t-h}^{t}  (\g u(t),w_t(t))_\G d\tau  \notag\\
& \lra \frac{1}{2}(\g u(t),w_t(t))_\G, \text{ as } h\lra 0.
\end{align}

Finally, by adding (\ref{3.9})-(\ref{3.10}) and by combining the results established in  (\ref{3.11})-(\ref{3.20}) we  can  pass to the limit as $ h\lra 0 $ to obtain the energy identity (\ref{energyeq}).

%%%%%%%%%%%%%%%%%%%%%%%%%%%%%

\subsection{Energy Inequality}
In order to complete the proof of Theorem~\ref{thm:exist}  in the case where $p>3$  it remains only to establish the energy inequalities  (\ref{energyineq1})-(\ref{energyineq2}) which are given in Proposition \ref{prop:energyineq} below. But, we first shall need some ancillary results   regarding the the sequences of approximate solutions 
 $\{u_N\}_1^\infty$ and $\{w_N\}_1^\infty$ which satisfy the conclusions of Corollary \ref {cor:converg}. 
 \begin{prop}\label{weak-1}
 Let  $\{u_N\}_1^\infty$ be the sequence of approximate solutions satisfying the conclusions of Corollary \ref {cor:converg}. Then, there is a subsequence, still labeled as  $\{u_N\}_1^\infty$, such that:
\begin{align}\label{3.22}
u'_N(t) \to u'(t) \text{ weakly in } L^2(\O), \text{ as }  N \to\infty, \text{  for all\, }  t\in[0,T].
\end{align} 
\end{prop} 
\begin{proof}
Let us first note that  the boundedness of the sequence $\{u_N\}_1^\infty$ in $L^{\infty} (0,T; X)$ implies that, the sequence 
$\{|u_N|^{p-1} u_N\}_1^\infty$ is bounded in $L^{\infty} (0,T; L^{\frac{p+1}{p}} (\O))$. Thus, on a subsequence labeled by $\{u_N\}_1^\infty$, we have
\[
|u_N|^{p-1} u_N \lra \xi \text{ \, weak}^*\text{ in } L^{\infty} (0,T; L^{\frac{p+1}{p}} (\O)).
\]
However, from the strong convergence in (\ref{converg:f}) we conclude (on a subsequence) that 
\[
|u_N|^{p-1} u_N \lra |u|^{p-1} u \text{ \,point-wise a.e. in } \O \times (0,T).
\]
Hence, $\xi =|u|^{p-1} u \text{  a.e. in } \O \times (0,T)$. That is, 
\begin{align}\label{3.23a}
|u_N|^{p-1} u_N \lra |u|^{p-1} u  \text{ \, weak}^*\text{ in } L^{\infty} (0,T; L^{\frac{p+1}{p}} (\O)).
\end{align}
From the first equation in (\ref{approx1}) along with (\ref{WaveGrad})-(\ref{WaveTrace}) and (\ref{3.23a}), we obtain, as $N \lra \infty$,
\begin{align}\label{3.23}
(u''_N,e_j)_\O \rightarrow -(\nabla u,\nabla e_j)_\O  + (w',\g e_j)_\G  -\int_\O|u|^{p-1}u e_jdx 
\text{\,  weak}^*\text{ in }L^\infty(0,T),  
\end{align}
for all $j\in \N$. By comparing (\ref{3.23}) with (\ref{2.24a}), it follows that 
\begin{align}\label{3.24}
 \frac{d}{dt} (u'_N,e_j)_\O \lra  \frac{d}{dt} (u', e_j)_\O
\text{ \, weak}^*\text{ in } L^\infty(0,T), \text{ for all  }  j \in \N.
\end{align} 
Since  $\chi_{[0,t]} \in L^1(0,T)$ for $t\in [0,T]$, then by integrating (\ref{3.24}) over $[0,t]$, we obtain 
\begin{align*}
(u'_N(t),e_j)_\O -(u'_N(0),e_j)_\O \lra  (u'(t), e_j)_\O -(u'(0), e_j)_\O, \text{ as  } N \lra \infty,
\end{align*} 
for all    $j \in \N$ and all $ t\in [0,T]$. By the strong convergence in (\ref{1.6}), it follows that
\begin{align}\label{3.25}
(u'_N(t),e_j)_\O  \lra  (u'(t), e_j)_\O, \text{ as  } N \lra \infty,
\end{align} 
for all  $j \in \N$ and all $ t\in [0,T]$.

Now, for any $\phi\in X$, there exists a sequence $\phi_k=\sum_{j=1}^k a_{k,j}e_j$ such that  
$\phi_k \rightarrow \phi$ strongly in $X$.  By linearity, one can replace  $e_j$ in \eqref{3.25} with $\phi_k$ to obtain
\begin{align}\label{3.26}
(u'_N (t),\phi_k)_\O \lra  (u'(t), \phi_k)_\O, \text{ as  }  N\lra \infty, \text{  for all } t\in [0,T].
\end{align} 
Thus, by using (\ref{3.26}) and the strong convergence of $\{\phi_k\}_{k=1}^\infty$ in $X$, we have for all $t\in [0,T]$:
\begin{align}\label{3.27}
\Big|( & u'_N (t),\phi)_\O -(u' (t),\phi)_\O \Big| \leq \Big|(u'_N (t),\phi)_\O-(u_N' (t),\phi_k)_\O \Big|  \notag \\ &+
\Big|(u'_N (t),\phi_k)_\O-(u' (t),\phi_k)_\O \Big|  + \Big|(u' (t),\phi_k)_\O- u' (t),\phi)_\O \Big|  \notag \\ 
& \leq  \norm{u'_N(t)}_2 \norm{\phi -\phi_k}_2 +\Big|(u'_N (t)-u'(t),\phi_k)_\O \Big| + 
\norm{u'(t)}_2 \norm{\phi -\phi_k}_2 \notag \\ 
& \leq C \norm{\phi -\phi_k}_2 +\Big|(u'_N (t)-u'(t),\phi_k)_\O \Big| \lra 0, \text{ as } N, k\lra \infty. 
\end{align} 
That is, for all $\phi \in X$,
\begin{align}\label{3.28}
(u'_N (t),\phi)_\O \lra  (u'(t), \phi)_\O, \text{  as  }  N \lra \infty, \text{  for all } t\in [0,T].
\end{align} 
Since the space $X$ is dense in $L^2(\O)$, then by a similar density argument as in (\ref{3.27}),  we conclude that  (\ref{3.28}) remains valid  for all $\phi \in L^2(\O)$, which completes the proof of the proposition.
\end{proof}

\begin{prop}\label{weak-2}
 The sequence of approximate solutions  $\{w_N\}_1^\infty$ satisfying the conclusions of Corollary \ref {cor:converg} also satisfies:
\begin{align}\label{3.29}
w'_N(t) \to w'(t) \text{ weakly in } L^2(\G), \text{ as }  N \to\infty, \text{  for all\, }  t\in[0,T].
\end{align} 
\end{prop} 

\begin{proof}
From the second equation in (\ref{approx1}) along with (\ref{Platelimits})-(\ref{2.33}) and (\ref{converg:a}), we have, as $N \lra \infty$,
\begin{align}\label{3.30}
(w''_N + \g u'_N ,\sigma_j)_\G \lra  & -(\Delta w,\Delta \sigma_j)_\G - (w' ,\sigma_j)_\G  \notag \\
& +\int_\G h(w)\sigma_jd\G \text{\,  weak}^*\text{ in }L^\infty(0,T),
\end{align}
for all $j\in \N$. By comparing (\ref{3.30}) with (\ref{2.38}), we conclude that 
\begin{align}\label{3.31}
 \frac{d}{dt} (w'_N + \g u_N ,\sigma_j)_\G \lra \frac{d}{dt} (w' + \g u ,\sigma_j)_\G
\text{ \, weak}^*\text{ in } L^\infty(0,T), \text{ for all  }  j \in \N.
\end{align} 
Again, as  $\chi_{[0,t]} \in L^1(0,T)$ for $t\in [0,T]$, then (\ref{3.31}) implies that
\begin{align}\label{3.32}
(w'_N (t)  + \g u_N(t) ,\sigma_j)_\G & - (w'_N (0) + \g u_N(0) ,\sigma_j)_\G \lra (w' (t)+ \g u (t),\sigma_j)_\G   \notag \\& - (w' (0)+ \g u (0),\sigma_j)_\G, \text{ as } N\lra \infty,
\end{align}
for all  $j \in \N$ and all $ t\in [0,T]$. By the strong convergence in (\ref{1.6}) and the continuity of trace operator 
$\g$, it follows that
\begin{align*}
(w'_N (t)  + \g u_N(t) ,\sigma_j)_\G \lra (w' (t)+ \g u (t),\sigma_j)_\G, \text{ as } N\lra \infty,
\end{align*}
for all  $j \in \N$ and all $ t\in [0,T]$. However, the strong convergence in (\ref{converg:trace}) yields,
\begin{align}\label{3.33}
(w'_N (t),\sigma_j)_\G \lra (w' (t),\sigma_j)_\G, \text{ as } N\lra \infty,
\end{align}
for all  $j \in \N$ and all $ t\in [0,T]$. Now, the rest of the proof goes exactly as in the proof of Proposition 
\ref{weak-1} by using a density argument.
\end{proof}

\begin{prop}\label{weak-3}
 Let $\{u_N\}_1^\infty$ and $\{w_N\}_1^\infty$ be the sequences of approximate solutions satisfying the conclusions of Corollary \ref {cor:converg}. Then, there are subsequences, still labeled as $\{u_N\}_1^\infty$ and $\{w_N\}_1^\infty$, such that, as $N\lra \infty$
\begin{align}\label{3.34}
\begin{cases}
 u_N(t)\lra  u(t) \text{ weakly in } L^{ p+1}(\O), \text{ a.e. } [0,T],  \vspace{.1in} \\
u_N(t)\lra u(t) \text{ weakly in } \H(\O), \text{ a.e. } [0,T], \vspace{.1in}  \\
w_N(t)\lra w(t) \text{ weakly in } H^2_0(\G), \text{ a.e. } [0,T]. \\
 \end{cases} 
\end{align} 
\end{prop} 

\begin{proof}
Since the sequence $\{u_N\}_1^\infty$ is bounded in $L^{\infty} (0,T; X)$, then in particular it is bounded in 
$L^1 (0,T; L^{ p+1}(\O))$. Thus, on a subsequence, it follows that 
\begin{align}\label{3.35}
u_N \lra  u \text{ \, weakly in }  L^1 (0,T; L^{ p+1}(\O)), \text{ as } N\lra \infty.
\end{align}
Thanks to the strong convergence in (\ref{converg:f}) which implies
\begin{align}\label{3.36}
u_N \lra  u \text{ \, strongly in }  L^1 (0,T; L^{ 2}(\O)), \text{ as } N\lra \infty.
\end{align}
Since $L^{ p+1}(\O) \subset L^{ 2}(\O) \subset L^{ \frac{p+1}{p}}(\O)$, then the first convergence in (\ref{3.34}) follows from Proposition \ref{prop6.2}  in the Appendix. The other two convergences in (\ref{3.34}) are also routine conclusions of  Proposition \ref{prop6.2}.
 \end{proof}

\begin{prop}\label{prop:energyineq}The limit functions $u$ and $w$ identified in Corollary~\ref{cor:converg} satisfy the energy inequalities \eqref{energyineq1} and \eqref{energyineq2} in the statement of Theorem~\ref{thm:exist}.
\end{prop}
\begin{proof}
From \eqref{apriori2} in the course of establishing the a priori estimates it was shown that each $u_N$ satisfies for all $t\in [0,T]$: 
\begin{align}\label{3.22}
&\mathscr{E}_N(t)+\int_0^t|w'_N(\tau)|_{2}^{2}d\tau=\mathscr{E}_N(0)+\int_0^t\int_{\Gamma}h(w_N(\tau))w'_N(\tau)d\Gamma d\tau,
\end{align}
where $\mathscr{E}_N(t)$ 
is  the positive energy of the system given by:
\begin{align*}
\mathscr{E}_N(t)=&\frac{1}{2}\left(\|u'_N(t)\|_2^2+\|\nabla u_N(t)\|_2^2+|w'_N(t)|_2^2 
+|\Delta w_N(t)|_2^2\right)  + \frac{1}{p+1}\|u_N(t)\|_{p+1}^{p+1}.
\end{align*}
By taking $H(w)=\int_0^{w} h(s)\,ds$ as the primitive of $h$, then (\ref{3.22}) becomes
\begin{align}\label{enineq2}
&\mathscr{E}_N(t)+\int_0^t|w'_N(\tau)|_2^2 d\t = \mathscr{E}_N(0)+\int_\G H(w_N(t)) d\G - \int_\G H(w_N(0))  d\G.
\end{align}
By defining the total energy by
\begin{align*}
E_N(t)=\mathscr{E}_N(t)-\int_\G H(w_N(t))d\G,
\end{align*}
we may recast  \eqref{enineq2} as
\begin{align}\label{enineq3}
E_N(t)+\int_0^t|w'_N(\tau)|_2^2 d\t= E_N(0).
\end{align} 
From the mean value theorem and the polynomial bound for $h$ in Remark \ref{rem1},  we have
\begin{align}\label{enineq4}
\Big| \int_\G \Big( H(w_N(t)) & -H(w(t)) \Big) d\G \Big| \leq C\int_G(1+|w_N(t)|^q+|w(t)|^q)|w_N(t)-w(t)|d\G \notag \\
&\leq C(1+|w_N(t)|_{2q}^q+|w(t)|_{2q}^q)|w_N(t)-w(t)|_2 \notag \\
&\leq C \sup_{t \in [0,T]} |\grad w_N(t)- \grad w(t)|_2 \longrightarrow 0, \text{ as } N \rightarrow \infty,
\end{align}
where we have used in (\ref{enineq4})  H\"older's inequality, the Sobolev Imbedding Theorem, and (\ref{converg:g}).
Hence,
\begin{align}\label{enineq5}
\lim_{N\lra\infty}\int_\G H(w_N(t))d\G=\int_\G H(w(t))d\G, \text{ for all } t\in[0,T].
\end{align}
Now, by taking the ``$\liminf_{N \rightarrow \infty}$" in (\ref{enineq3}), we obtain
\begin{align}\label{3.43}
\liminf_{N \rightarrow \infty} E_N(t)+ \liminf_{N \rightarrow \infty} \int_0^t|w'_N(\tau)|_2^2 d\t \leq \liminf_{N \rightarrow \infty} E_N(0) =E(0),
\end{align} 
were we have used (\ref{enineq5}) and the strong convergence in (\ref{1.6})-(\ref{1.7}).

Using the weak lower-semicontinuity of norms, Fatou's Lemma, and (\ref{enineq5}) along with Proposition \ref{weak-1}-Proposition \ref{weak-3}, we obtain for almost all $t\in[0,T]$,
\begin{align}\label{3.44}
\liminf_{N  \rightarrow \infty} \,& E_N(t) + \liminf_{N \rightarrow \infty} \int_0^t|w'_N(\tau)|_2^2 d\t  \geq \liminf_{N \rightarrow \infty} \mathscr{E}_N(t) +  \int_0^t|w' (\tau)|_2^2 d\t \notag\\ &- \lim_{N \rightarrow \infty} \int_\G H(w_N(t))d\G 
 \geq  \mathscr{E}(t)+  \int_0^t|w' (\tau)|_2^2 d\t -\int_\G H(w(t)) d\G. 
\end{align} 
Combining (\ref{3.43}) with (\ref{3.44}), we obtain 
\begin{align}\label{3.45} 
\mathscr{E}(t)+  \int_0^t|w' (\tau)|_2^2 d\t -\int_\G H(w(t)) d\G \leq E(0) \text{ a.e. } [0,T], 
\end{align} 
which is precisely the desired energy inequality (\ref{energyineq2}).

Finally, the  energy inequality (\ref{energyineq1}) is easily obtained after showing 
\begin{align}\label{3.46}
\lim_{N\lra\infty}\int_0^t\int_\G h(w_N(\tau))w_N'(\tau)d\G d\tau=\int_0^t\int_\G h(w(\tau))w'(\tau)d\G d\tau.
\end{align}
The proof of \eqref{3.46} is similar to the proof of \eqref{enineq5}, and thus it is omitted. 
\end{proof}

\section{Global Existence}\label{S4}
\noindent{}This section is devoted to prove the existence of global solutions as described in Theorem~\ref{thm:global}. As in \cite{AR2,GR,PRT-p-Laplacain} and other works, it is the case here that either a given solution $(u, w)$ must exist globally in time or else one may find a value of $T_0$ with $0<T_0<\infty$ so that 
\begin{align}\label{4.1}
\limsup_{t\to T_0^-} \Big( \mathscr{E}(t)+  \int_0^t|w_t (\tau)|_2^2 d\t \Big)=\infty,
\end{align}
where,  $\mathscr{E}(t)=\frac{1}{2}\left(\|u_t(t)\|_2^2+\|\nabla u(t)\|_2^2+|w_t(t)|_2^2+|\Delta w(t)|_2^2\right)+\frac{1}{p+1}\|u(t)\|_{p+1}^{p+1}.$

By demonstrating a bound on the energy
\[
\mathscr{E}(t)+  \int_0^t|w_t (\tau)|_2^2 d\t
\]
on every interval $[0,T]$ which is dependent only upon $T$ and the positive initial  energy  $\E(0)$, we shall show that the scenario in \eqref{4.1} cannot occur as the argument is bounded on any finite interval.  This bound is possible provided  the source term acting on the plate is essentially linear.  Indeed, this assertion  is contained in the following proposition.
\begin{prop}\label{prop:global}
Let $(u,w)$ be a weak solution of \eqref{PDE} on $[0,T]$ as furnished by Theorem~\ref{thm:exist}.  
\begin{itemize}
\item If $q=1$, then for all $t\in [0,T]$, $(u,w)$ satisfies
\begin{align}\label{4.2}
\mathscr{E}(t)+ \int_0^t |w(\t)|_2^2 d\tau \leq C(T,\mathscr{E}(0)),
\end{align}
where $0<T<\infty$ is aribitrary. 
\item If $q>1$, then the bound  in \eqref{4.2} holds  for all $0\leq t\leq T'$, where $0<T'\leq T$ and $T'$ depending upon $T$ and $\mathscr{E}(0)$.
\end{itemize}
\end{prop}
\begin{proof}
Recall the energy inequality in \eqref{energyineq1}:
\begin{align}\label{g1}
\mathscr{E}(t)+\int_0^t|w_t (\tau)|_2^2 d\t \leq \mathscr{E}(0)+\int_0^t\int_{\G}h(w)w_td\G d\tau.
\end{align}
By noting the polynomial bound on $h$ in Remark \ref{rem1} with $q=1$ along with H\"older's and Young's inequalities, we have:
\begin{align}\label{g2}
\left|\int_0^t\int_{\G}h(w)w_td\G d\tau\right|&\leq C\int_0^t\int_\G(|w(\tau)|+1)w_t(\tau)d\G d\tau \notag \\
&\leq C\int_0^t  (|w(\tau)|_2+1) |w_t(\tau)|_2 d\tau \notag \\
&\leq\frac{1}{2}\int_0^t |w_t(\tau)|_2^2d\tau+C\int_0^t |w(\tau)|_2^2d\tau+CT \notag \\
&\leq\frac{1}{2}\int_0^t |w_t(\tau)|_2^2d\tau+C\int_0^t\mathscr{E}(\tau)d\tau+CT,
\end{align}
where the constant $C$ in \eqref{g2} depends on $|\G|$, the Lebesgue measure of $\G$.
Combining \eqref{g1} and \eqref{g2} yields,
\begin{align}\label{g2.5}\mathscr{E}(t)+\frac{1}{2}\int_0^t |w_t(\tau)|_2^2d\tau \leq 
\mathscr{E}(0) +CT+C\int_0^t\mathscr{E}(\tau)d\tau.
\end{align}
In particular, 
\begin{align}\label{g3}
\mathscr{E}(t)\leq\mathscr{E}(0) +CT +C\int_0^t\mathscr{E}(\tau)d\tau\text{ for } t\in[0,T].
\end{align}
By Gronwall's inequality, we conclude that
\begin{align}\label{g4}
\mathscr{E}(t)\leq (\mathscr{E}(0)+CT)e^{CT}\text{ for } t\in[0,T],
\end{align}
where $T>0$ is arbitrary.  Combining \eqref{g2.5} and \eqref{g4}, the desired result in (\ref{4.2}) follows.

\medskip

\noindent{}Now, if $q>1$, we  appeal to the polynomial bonud on $h$ in Remark \ref{rem1}  along with H\"older's and Young's inequalities to obtain:
\begin{align}\label{g5}
\left|\int_0^t\int_{\G}h(w)w_td\G d\tau\right|&\leq C\int_0^t\int_\G(|w(\tau)|^q+1)w_t(\tau)d\G d\tau \notag \\
&\leq C\int_0^t  (|w(\tau)|^q_{2q}+1) |w_t(\tau)|_2 d\tau \notag \\
&\leq\frac{1}{2}\int_0^t |w_t(\tau)|_2^2d\tau+C\int_0^t |\Delta w(\tau)|_{2}^{2q} d\tau+CT \notag \\
&\leq\frac{1}{2}\int_0^t |w_t(\tau)|_2^2d\tau+C\int_0^t\mathscr{E}(\tau)^q d\tau+CT. 
\end{align}
Combining \eqref{g1} and \eqref{g5} yields
\begin{align}\label{4.9}
\mathscr{E}(t)+\frac{1}{2}\int_0^t\|w_t(\tau)\|_2^2d\tau\leq\mathscr{E}(0)+CT +C\int_0^t\mathscr{E}(\tau)^q d\tau.
\end{align}
In particular,
\begin{align}\label{g6}\mathscr{E}(t)\leq\mathscr{E}(0) +CT +C\int_0^t\mathscr{E}(\tau)^q d\tau.
\end{align}
By using a standard comparison theorem, \eqref{g6} yields that $\mathscr{E}(t)\leq z(t)$, where $z(t)=\left[(\mathscr{E}(0)+CT)^{1-q}-C(q-1)t\right]^{\frac{-1}{q-1}}$ is the solution of the Volterra integral equation
\begin{align*}
z(t)=\mathscr{E}(0)+CT+ C\int_0^tz(s)^qds.
\end{align*}
Since $q>1$, $z(t)$ blows up at the finite time $T_1=\frac{1}{C(q-1)}(\mathscr{E}_0+CT)^{1-q}$.  Note that $T_1$ depends on initial energy $\mathscr{E}(0)$ and the original existence time, $T$.  Nonetheless, if we choose $T'=\min\{T,\frac{1}{2}T_1\}$, then
\begin{align}\label{g7}
\mathscr{E}(t)\leq z(t)\leq C_0:=\left[(\mathscr{E}(0)+C_T)^{1-q}-C(q-1)T'\right]^{\frac{-1}{q-1}} <\infty,
\end{align}
for all $t\in[0,T']$.  Finally, we combine \eqref{4.9} and \eqref{g7} to  conclude the second statement of the 
proposition.
\end{proof}

\section{Continuous Dependence on Initial Data}\label{S5}
\noindent{}In this section, we provide the proof to Theorem~\ref{thm:CDID} in the case $1\leq p\leq 3$, where the   bound  (\ref{4.2}) is crucial in the proof.
\begin{proof}
Let $U_0=(u_0,w_0,u_1,w_1)\in H=H^1_{\G_0}(\O)\times H_0^2(\G)\times L^2(\O)\times L^2(\G)$.  Assume that $\{U_0^n=(u_0^n,w_0^n,u_1^n,w_1^n): n\in \N\}$ is a sequence of initial data that satisfies:
\begin{align}\label{c1}
U_0^n\lra U_0\text{ in }H\text{ strongly as }n\lra\infty.
\end{align}
Let $\{(u^n,w^n)\}$ and $(u,w)$ be the weak solutions to \eqref{PDE} defined on $[0,T]$ in the sense of Definition~\ref{def:weaksln}, corresponding to the initial data $\{U_0^n\}$ and $\{U_0\}$, respectively.  First, we show that the local existence time $T$ can be taken independent of $n\in\N$. To see this, we recall that the local existence time provided by Theorem \ref{thm:exist} for the solution $(u,w)$ depends on the initial energy $\E(0)$.
Due to the strong convergence of $U_0^n\lra U_0$, then the local existence  time $T$  for the solutions 
$\{(u^n, w^n)\}$ and $(u,w)$ can be chosen independent of $n\in\N$. Moreover, in view of (\ref{4.2}), $T$ can be taken arbitrarily large in the case when $q=1$. However, in the case when  $q>1$, we select the local existence time to be $T=T'$, where $T'$ is as given in Proposition \ref{prop:global} (which is also uniform in $n$). In either case, it follows from  (\ref{4.2}) that there exists $R>0$ such that, for all $n\in\N$ and all $t\in [0,T]$ (where $T>0$ is independent of $n$):
\begin{align}\label{5.2}
\begin{cases}
\mathscr{E}(t)+ \int_0^t |w(\t)|_2^2 d\tau \leq R, \vspace{.1in} \\
\mathscr{E}^n(t)+ \int_0^t |w^n(\t)|_2^2 d\tau \leq R,
\end{cases}
\end{align}
where $\mathscr{E}^n(t)=\frac{1}{2}\left(\|u^n_t(t)\|_2^2+\|\nabla u^n(t)\|_2^2+|w^n_t(t)|_2^2+|\Delta w^n(t)|_2^2\right)+\frac{1}{p+1}\|u^n(t)\|_{p+1}^{p+1}$. 

Now, put $y^n(t)=u(t)-u^n(t)$, $z^n(t)=w(t)-w^n(t)$, and
\begin{align}\label{c3}
\tilde{\mathscr{E}}^n(t)=\frac{1}{2}\left(\|y^n_t(t)\|_2^2+\|\nabla y^n(t)\|_2^2+|z^n_t(t)|_2^2+|\Delta z^n(t)|_2^2\right),
\end{align}
for $t\in[0,T]$.  We aim to show $\tilde{\mathscr{E}}_n(t)\lra0$ uniformly on $[0,T]$.
\\
\\
From Definition~\ref{def:weaksln}, then $y^n$ and $z^n$ satisfy:

\begin{align}\label{c4}
 (y^n_{t}(t),\phi(t))_\O  & -(y^n_t(0),\phi(0))_\O-\int_0^t ( y^n_t(\tau), \phi_t(\tau) )_\O d\tau
\notag\\ &+\int_0^t ( \nabla y^n(\tau), \nabla\phi(\tau) )_\O d\tau 
-\int_0^t ( z^n_t(\tau), \g\phi(\tau))_\G d\tau \notag\\ & +\int_0^t\int_\Omega \Big(|u(\tau)|^{p-1}u(\tau)-|u^n(\tau)|^{p-1}u^n(\tau)\Big)\phi(\tau) dxd\tau=0,
\end{align}
\begin{align}\label{c5}
 (z^n_t(t) &+ \g y^n(t),  \psi(t))_\G -(z^n_t(0)+ \g y^n(0,\psi(0))_\G-\int_0^t ( z^n_t(\tau), \psi_t(\tau))_\G d\tau \notag \\
& -\int_0^t ( \g y^n(\tau) , \psi_t(\tau))_\G d\tau+\int_0^t ( \Delta z^n(\tau), \Delta\psi(\tau) )_\G d\tau \notag \\
&+\int_0^t (z^n_t(\tau), \psi(\tau) )_\G d\tau=\int_0^t\int_{\G} \Big(h(w(\tau))-h(w^n(\tau)) \Big)\psi(\tau) d\G d\tau, 
\end{align}
where $\p$ and $\psi$ are proper test functions as described in Definition \ref{def:weaksln}. 

As we demonstrated in the proof of the energy identity in Section \ref{S3}, we can replace $\phi(\tau)$ by $D_h y(\tau)$ in \eqref{c4} and $\psi(\tau)$ by $D_hz(\tau)$ in \eqref{c5}, for any $\tau\in[0,T]$.  By using  similar arguments as in the proof of the energy identity \eqref{energyeq}, we can pass to the limit as $h\lra 0$ to deduce the identity:
\begin{align}\label{c6}
\tilde{\mathscr{E}}^n(t) & +\int_0^t|z^n (\t)|_2^2d\tau+\int_0^t\int_\Omega \Big(|u(\tau)|^{p-1}u(\tau)-|u^n(\tau)|^{p-1}u^n(\tau)\Big) y_t^n(\tau) dxd\tau  \notag \\
& = \tilde{\mathscr{E}}^n(0)+\int_0^t\int_{\G} \Big(h(w(\tau))-h(w^n(\tau)) \Big) z_t^n(\tau) d\G d\tau.
\end{align}
We first estimate the term coming from the source acting on the wave equation. By recalling the bounds in Remark \ref{rem1} and by using  H\"older's and Young's Inequalities, one has
\begin{align}\label{c7}
\Big|\int_0^t\int_\Omega& \Big(|u(\tau)|^{p-1}u(\tau)-|u^n(\tau)|^{p-1}u^n(\tau) \Big)y_t^n(\tau) dxd\tau \Big| \notag \\
&\leq C \int_0^t\int_\Omega \left(|u(\tau)|^{p-1}+|u^n(\tau)|^{p-1}\right)|u(\tau)-u^n(\tau)||y^n_t(\tau)|dxd\tau  \notag\\
&\leq C\int_0^t  (\| u(\tau)\|_{3(p-1)}^{p-1}+\|u^n(\tau)\|_{3(p-1)}^{p-1}   \|u(\tau)-u^n(\tau)\|_6\|y^n_t(\tau)\|_2d\tau  \notag \\ &\leq C_R \int_0^t (\|\nabla y^n(\tau)\|_2^2+\|y^n_t(\tau)\|_2^2) d\tau
\leq C_R\int_0^t\tilde{\mathscr{E}}^n(\tau)d\tau,
\end{align}
where we have used in \eqref{c7} the assumption $1\leq p\leq 3$, the Sobolev Imbedding Theorem, and the bounds in (\ref{5.2}).

In a similar manner, we can estimate the term coming from the source acting on the plate and obtain 
\begin{align}\label{c9}
\Big| \int_0^t\int_{\G} \Big(h(w(\tau))-h(w^n(\tau)) \Big) z_t^n(\tau) d\G d\tau \Big| & \leq\ C_R \int_0^t|\Delta z^n(\tau)|_2^2d\tau  \notag  \\
& \leq C_R \int_0^t\tilde{\mathscr{E}}^n(\tau)d\tau.
\end{align}
By combining \eqref{c6}-\eqref{c9}, we conclude
\begin{align}\label{c10}
\tilde{\mathscr{E}}^n(t) +\int_0^t|z^n (\t)|_2^2d\tau \leq\tilde{\mathscr{E}}^n(0)+C_R\int_0^t\tilde{\mathscr{E}}^n(\tau)d\tau.
\end{align}
In particular, Gronwall's inequality yields
\begin{align}\label{c11}
\tilde{\mathscr{E}}^n(t)\leq \tilde{\mathscr{E}}^n(0)e^{C_RT}, \text{ for all  } t\in [0,T].
\end{align}
Since $ \tilde{\mathscr{E}}^n(0) \lra 0$, as $n\lra \infty$, then $\tilde{\mathscr{E}}_n(t)\lra0$ uniformly on $[0,T]$, completing the proof.
\end{proof}

\begin{rmk}
Corollary~\ref{cor:uni} follows immediately from Theorem~\ref{thm:CDID}.  Its proof is outlined below.
\end{rmk}

\begin{proof}
Let $(u,w)$ and $(\hat{u},\hat{w})$ be two weak solutions to \eqref{PDE} defined on $[0,T]$ in the sense of Definition~\ref{def:weaksln} with the same initial data $U_0=(u_0,w_0,u_1,w_1)\in H$, where 
$H=H^1_{\G_0}(\O)\times H_0^2(\G)\times L^2(\O)\times L^2(\G)$.  Put:  $\hat{y}(t)=u(t)-\hat{u}(t)$, $\hat{z}(t)=w(t)-\hat{w}(t)$, and 
\begin{align}\label{u1}
\hat{\mathscr{E}}(t)=\frac{1}{2}\left(\hat{y}'(t)\|_2^2+\|\nabla\hat{y}(t)\|_2^2+|\hat{z}'(t)|_2^2+|\Delta\hat{z}(t)|_2^2\right).
\end{align}
Then, in the same manner in obtaining the identity \eqref{c6}, we have
\begin{align}\label{u4}
\hat{\mathscr{E}}(t)+\int_0^t|\hat{z}(\t) |_2^2d\tau &+\int_0^t\int_\Omega \Big(|u(\tau)|^{p-1}u(\tau)-|\hat{u}(\tau)|^{p-1}\hat{u}(\tau) \Big)y_t^n(\tau) dxd\tau \notag \\
& \leq\int_0^t\int_{\G} \Big(h(w(\tau))-h(\hat{w}(\tau)) \Big)z_t^n(\tau) d\G d\tau
\end{align}
Similar estimates as in \eqref{c7}-\eqref{c9} yield,
\begin{align}\label{u5}
\hat{\mathscr{E}}(t) +\int_0^t|\hat{z}(\t) |_2^2d\tau \leq C\int_0^t\hat{\mathscr{E}}(\tau)d\tau,
\end{align}
which implies by Gronwall's inequality that $\hat{\mathscr{E}}(t)=0$  for all $t\in [0,T]$. Hence,  
$(u,w)=(\hat{u},\hat{w})$.  
\end{proof}

\section{Appendix}
The following auxiliary results were invoked in the proof of the main theorem of existence of local weak solutions. These results appeared in various references (we refer the reader to \cite{PRT-p-Laplacain,RS1,RW} for instance). We list them here for sake of convenience.
 
\begin{prop}[Prop. A.1 in \cite{PRT-p-Laplacain}]\label{prop:prodrule} Let $H$ be a Hilbert space and $X$ be a Banach space such that $X \subset H \subset X'$ where each injection is continuous with dense range. If 
	\begin{align*}
	\begin{cases}
	f\in L^2(0,T;H),\quad f'\in L^2(0,T;X'),  \vspace{.1in}\\
	g \in L^2(0,T;X),\quad g' \in L^2(0,T;H),
	\end{cases}
	\end{align*}
	then the map $t\mapsto (f(t),g(t))_H$ coincides with an absolutely continuous on $[0,T]$ and 
	\begin{align*}
	\frac{d}{dt}(f(t),g(t))_H = \langle f'(t),g(t)\rangle_{X',X} + (f(t),g'(t))_H\text{ a.e. }[0,T].
	\end{align*}
\end{prop}

\begin{prop}[{Prop. A.2 in \cite{PRT-p-Laplacain}}]
\label{prop6.2}
Let $H$ be a Hilbert space and $X$ be a Banach space such that  $X\subset H\subset X'$ where with each injection is continuous with dense range. Suppose $X'$ is separable and  $\{u_N\}_1^\infty$  is a  sequence in $L^1(0,T;X)$ satisfying:
\begin{align*}
\begin{cases}
u_N\to u \text{  weakly in  }  L^1(0,T; X), \vspace{.1in} \\
u_N\to u  \text{  strongly in  }  L^1 (0,T; H),
\end{cases}
\end{align*}
as $N\to \infty$. Then,  there exists a subsequence of $\{u_N\}_1^\infty$  (again reindexed by $N$) such that
\begin{align*}
u_N(t) \to u(t) \text{ weakly in }
X \text{ a.e. } [0,T], \text{ as }  N \to\infty.
\end{align*}
\end{prop}

\bibliographystyle{abbrv}
\bibliography{mohbib-new}
\end{document}